\documentclass[times]{nlaauth}

\usepackage{moreverb}
\usepackage[colorlinks,bookmarksopen,bookmarksnumbered,citecolor=red,urlcolor=red]{hyperref}

\newcommand\BibTeX{{\rmfamily B\kern-.05em \textsc{i\kern-.025em b}\kern-.08em
T\kern-.1667em\lower.7ex\hbox{E}\kern-.125emX}}

\def\volumeyear{2010}

\runningheads{SAIBABA ET AL}{Randomized algorithms for GHEP}

\title{Randomized algorithms for Generalized Hermitian Eigenvalue Problems with application to computing Karhunen-Lo\`{e}ve expansion}

\author{Arvind K. Saibaba\corrauth, Peter K. Kitanidis and Jonghyun Harry Lee}

\address{Institute for Computational and Mathematical Engineering, Huang Building
475 Via Ortega, Stanford University, California-94305 }

\corraddr{Institute for Computational and Mathematical Engineering, Huang Building
475 Via Ortega, Stanford University, California-94305}

\usepackage[english]{babel}
\usepackage{amsmath,amssymb,amsthm}
\usepackage{graphicx}
\usepackage{algorithm,algorithmic}
\usepackage{multirow}

\newtheorem{propos}{Proposition}
\newtheorem{theorem}{Theorem}

\newcommand{\norm}[2]{\lVert #1 \rVert_{#2}}
\newcommand{\normtwo}[1]{\lVert #1 \rVert_2}
\newcommand{\normB}[1]{\lVert #1 \rVert_B}

\newcommand{\define}{\stackrel{\text{def}}{=}}
\newcommand{\prior}{\Gamma_\text{prior}}
\newcommand{\noise}{\Gamma_\text{noise}}
\newcommand{\bigO}{{\cal{O}}}
\newcommand{\post}{\Gamma_\text{post}}

\newcommand{\bx}{\textbf{x}}

\begin{document}

\begin{abstract}
We describe randomized algorithms for computing the dominant eigenmodes of the Generalized Hermitian Eigenvalue Problem (GHEP) $Ax=\lambda Bx$, with $A$  Hermitian and $B$ Hermitian and positive definite. The algorithms we describe only require forming operations $Ax$, $Bx$ and $B^{-1}x$ and avoid forming square-roots of $B$ (or operations of the form, $B^{1/2}x$ or $B^{-1/2}x$). We provide a convergence analysis  and a posteriori error bounds that build upon the work of~\cite{halko2011finding,liberty2007randomized,martinsson2011randomized} (which have been derived for the case $B=I$). Additionally, we derive some new results that provide insight into the accuracy of the eigenvalue calculations. The error analysis shows that the randomized algorithm is most accurate when the generalized singular values of $B^{-1}A$ decay rapidly. A randomized algorithm for the Generalized Singular Value Decomposition (GSVD) is also provided. Finally, we demonstrate the performance of our algorithm on computing the Karhunen-Lo\`{e}ve expansion, which is a computationally intensive GHEP problem with rapidly decaying eigenvalues.
\end{abstract}


\maketitle

\section{Introduction}
Consider the Generalized Hermitian Eigenvalue Problem (GHEP)

\begin{equation}
\label{eqn:ghep}
Ax = \lambda Bx
\end{equation}
where, $B$ is Hermitian positive definite and $A$ is Hermitian. The analysis is also relevant if $B$ is not positive definite. In that case, if $B$ is not positive definite but some combination $(\alpha A + \beta B)$ is positive definite, then
the transformed problem $ Ax  = \theta (\alpha A + \beta B)x $ has eigenvalues $\theta_i = \lambda_i/(\alpha \lambda_i + \beta)$ and has the same eigenvectors as Equation~\eqref{eqn:ghep}.

We can transform the GHEP into a Hermitian Eigenvalue Problem (HEP), which is of the form $Mx = \lambda x$ for matrices $M$ positive semidefinite. Since $B$ is positive definite, it has a Cholesky Decomposition $B=LL^*$. Define $y = L^*x$ and multiplying both sides of Equation~\eqref{eqn:ghep} by $L^{-1}$, we have

\begin{equation}
\label{eqn:hep}
L^{-1}AL^{-*}L^{*}x = \lambda L^*x \qquad \Rightarrow \qquad L^{-1}AL^{-*}y = \lambda y
\end{equation}
which is a HEP and hence, any algorithm for HEPs can be used to solve GHEPs. However, computing the Cholesky decomposition is not computationally feasible for several matrices. It should be noted that this type of transformation~\eqref{eqn:hep} can be derived for any definition of square root of a matrix. Although several algorithms exist for computing the square root of a matrix, or performing matrix-vector products (henceforth, called matvecs) $B^{1/2}x$ or $B^{-1/2}x$, their application to large-scale problems is not always efficient. Another transformation, $B^{-1}Ax = \lambda x $ makes the problem into a regular eigenvalue problem. Even though $A$ and $B$ are Hermitian,  in general $B^{-1}A$ will not be Hermitian. We will focus our attention on problems for which computing Cholesky decomposition (or any other square root, for that matter) is too expensive to compute explicitly.

The key idea that we will exploit in this paper is the fact that, while $B^{-1}A$ is not Hermitian, it is Hermitian with respect to another inner product, the  $B$-inner product which we will define shortly. This property has previously been exploited by Krylov subspace based eigensolvers~\cite{grimes1994shifted,saad1992numerical}. An added advantage to using $B$-inner products is that the resulting eigenvectors are now $B$-orthonormal. There are several methods for solving the GHEP~\eqref{eqn:ghep}. These include approaches based on power and inverse iteration methods, Lanczos based methods and Jacobi-Davidson method. For a good review on this material, please refer to~\cite[chapter 5]{bai1987templates} and~\cite{saad1992numerical}. A good survey of existing software for sparse eigenvalue problems including GHEP is available at~\cite{str-6}.

Randomized algorithms have been developed for approximately computing a low-rank decomposition when the singular values decay rapidly (for a comprehensive review, see~\cite{halko2011finding}). After computing the approximate low-rank decomposition, an additional post-processing step can be performed to compute the approximate singular value decomposition. For Hermitian operators, this post-processing can be modified to obtain an approximate eigenvalue decomposition as well. The randomized SVD algorithm can be applied to the matrix $C\define B^{-1}A$ to obtain an approximate singular value decomposition. However, applying the algorithm directly to the matrix $C$, will result in singular vectors that are orthonormal but not $B$-orthonormal. A symmetric low-rank decomposition is highly desirable in many application. As a result, we would like to develop square-root free variants of the randomized SVD algorithm to compute the dominant eigenmodes of the GHEP.

The algorithms described in this paper are useful when it is necessary to quickly compute an approximation to the largest eigenmodes. The only requirement is availability of fast ways to compute $Ax$, $Bx$ and $B^{-1}x$ and it avoids computations of the form $B^{1/2}x$ and $B^{-1/2}x$. As a result, this algorithm is applicable to very general settings. The randomized algorithms are often faster, are quite robust and accompanied by theoretical guarantees.  The error analysis suggests that the algorithms are most accurate when the (generalized) singular values of $B^{-1}A$ decay rapidly. Moreover, the low-rank decompositions can be produced to any user defined tolerance, which allows the user to trade-off between computational cost and accuracy. While it is certainly true that under the same settings, Krylov subspace methods often are more accurate especially for systems of the form~\eqref{eqn:ghep} with rapidly decaying eigenvalues, randomized schemes are numerically robust and allow freedom in exploiting gains from parallelism and block matrix-vector products. As a result, randomized algorithms are well suited to computationally intensive problems and modern computing environments. For example, when efficient block methods to compute $Ax$, $Bx$ or $B^{-1}x$ exist, they can be used to significantly speed up calculations. Finally, Krylov subspace methods must be often accompanied by sophisticated algorithms to monitor restart, orthogonality and loss of precision. Randomized algorithms, on the other hand, are straightforward to implement in very few lines of code which are transparent to read. To summarize, one must weigh the trade-offs between using randomized algorithms which are low-cost, easy to implement and robust and using Krylov subspace based methods that are capable of higher accuracy but are often, much more expensive. A further discussion of the suitability of randomized algorithms to high performance computing is available in~\cite{halko2011finding,bui2012extreme}.

In addition to describing the randomized algorithms, a rigorous error analysis is also provided that closely follows the proof techniques developed in~\cite{liberty2007randomized,halko2011finding}. Furthermore, we provide computable a posteriori error bounds on 1) the approximate low-rank representation, and 2) the error between the true and the approximate eigenvalues (and eigenvectors) as a function of the low-rank representation error. To the best of our knowledge, the latter result is not available even for the case $B=I$. We also provide a randomized algorithm for the Generalized Singular Value Decomposition (GSVD). We demonstrate the performance of our algorithms on a challenging application - computing the dominant eigenmodes of the Karhunen-Lo\`{e}ve expansion.

\section{ Algorithms}\label{sec:randomized}

\begin{algorithm}[!ht]
\begin{algorithmic}[1]
\REQUIRE  matrices $A \in \mathbb{C}^{n\times n}$, and $\Omega  \in \mathbb{R}^{n\times (k+p)}$  a Gaussian random matrix. Here $k$ is the desired rank, and $p$ is an oversampling factor.
\STATE  Compute $Y = A\Omega$, and compute QR factorization $Y=QR$
\STATE  Form $ B = Q^*A $
\STATE  Compute SVD of the small matrix $B = \tilde{U}\Sigma V^*$
\STATE Form the orthonormal matrix $U = Q\tilde{U}$
\RETURN $U$, $\Sigma$ and $V$ that satisfy $A \approx U\Sigma V^*$
\end{algorithmic}
\caption{Randomized SVD}
\label{alg:randsvd}
\end{algorithm}

Let us first review the randomized SVD algorithm that is described in~\cite{halko2011finding} to compute the rank-$k$ decomposition for any matrix $A\in \mathbb{C}^{m\times n}$ that has rapidly decaying singular values. The algorithms proceed by computing a matrix $Q$ whose columns form a basis for the approximate range of $A$. This is accomplished by forming matvecs of A with random vectors drawn from an i.i.d. Gaussian distribution. The matrix $Q$ satisfies the bound $\normtwo{(I-QQ^*)A} \leq \varepsilon$. This is summarized in Algorithm~\ref{alg:randsvd}. If $A$ is Hermitian, and we have found a $Q$ that satisfies  $\normtwo{(I-QQ^*)A} \leq \varepsilon$, then it can be shown that $\normtwo{A-QQ^*AQQ^*} \leq 2\varepsilon$. With this observation, an additional step can be performed to compute a Hermitian eigenvalue decomposition. The smaller matrix $T=Q^*AQ$ is formed and its eigendecomposition $S\Lambda S^*$ is computed and $A \approx U \Lambda U^*$, where $U=QS$. This is the two pass version of the algorithm to compute largest eigenvalues and corresponding eigenvectors. A second round of matrix-vector products involving $A$ (to compute $T=Q^*AQ$) can be avoided by using the information contained in $Y, Q$ and $\Omega$. This is known as a single pass algorithm. This is summarized in Algorithm~\ref{alg:randhep}. For further details regarding the aforementioned algorithms, the reader is referred to~\cite{halko2011finding}.

\begin{algorithm}[!ht]
\begin{algorithmic}[1]
\REQUIRE  matrices $A \in \mathbb{C}^{n\times n}$, and $\Omega  \in \mathbb{R}^{n\times (k+p)}$  a Gaussian random matrix. Here $k$ is the desired rank, and $p$ is an oversampling factor.
\STATE  Compute $Y = A\Omega$, and compute QR factorization $Y=QR$
\STATE  Form $ T = Q^*AQ $ (two-pass) or $T \approx (Q^*Y)(Q^*\Omega)^{-1} $ (single-pass)
\STATE  Compute EVD of the small matrix $T = S\Lambda S^*$
\STATE Form the orthonormal matrix $U = QS$
\RETURN $U$, $\Lambda$ that satisfy $A \approx U\Lambda U^*$
\end{algorithmic}
\caption{Randomized eigenvalue decomposition}
\label{alg:randhep}
\end{algorithm}

 The main difference is that we replace the inner-product with a $B$-inner product and as a result, we no longer maintain an orthonormal basis $Q$ but a $B$-orthonormal basis. Here we summarize some basic results about $B$-inner products and the resulting vector and matrix norms. The $B$-inner product is defined as $\langle x,y \rangle_B \define y^*Bx$ and the $B$-norm $\normB{x} \define \sqrt{x^*B x}$. It satisfies the following inequality, (see, for example~\cite{meerbergen1998theoretical})
\begin{equation}
\label{eqn:Bnorm2normineq}
\frac{\normtwo{x}^2 }{\normtwo{B^{-1}}} \quad \leq \quad \normB{x}^2 \quad  \leq \quad \normtwo{x}^2 \normtwo{B}
\end{equation}

Let us define the matrix $C \define B^{-1}A$. It can be verified that $C$ is self-adjoint with respect to the $B$-inner product, i.e. $\langle Cx,y\rangle_B = \langle x,Cy \rangle_B$. The $B$-norm of a matrix is defined as an induced vector norm $\normB{M}= \max_{\normB{x} = 1} \normB{Mx}$. We will make use of this fact to derive randomized algorithms for GHEP that produces a Hermitian low-rank decomposition.
It can be verified that for any matrix $M$, making the transformation $y = B^{1/2}x$, we have that

\begin{equation}\label{eqn:prop2}
\normB{M} = \max_{x} \frac{\normB{Mx}}{\normB{x}} = \max_{y} \frac{\normtwo{B^{1/2}MB^{-1/2}y} }{\normtwo{y}} =   \normtwo{B^{1/2}MB^{-1/2}}
 %\normB{M} = \max_{x} \frac{\normB{Mx}}{\normB{x}} \leq \max_x \frac{\normB{Mx}}{\normtwo{x}} \sqrt{\normtwo{B^{-1}}}  = \sqrt{\normtwo{B^{-1}}} \normtwo{B^{1/2}M}
\end{equation}

For the error analysis, we will need a generalized notion of singular values, defined as follows
\begin{equation}\label{eqn:gensingular} \sigma_{B} (M) = \left\{\mu \left| \mu \text{ are the stationary points of } \frac{\normB{Mx}}{\normtwo{x}} \right.\right\}\end{equation}
This definition is similar to~\cite[definition 3]{van1976generalizing} with $S=B$ and $T=I$. This results in a following decomposition of the form
\[ M = U\Sigma_B V^* \qquad U^*BU=I \qquad  V^*V=I\]
and $\Sigma_B = \text{diag}\{\sigma_{B,1},\dots,\sigma_{B,n}\}$ are the generalized singular values. They have a subscript to distinguish them from the singular values defined in the regular sense. The existence of this decomposition is guaranteed by~\cite[Theorem 3]{van1976generalizing}.

\subsection{Approximating the range of $C$}

 The key step of the algorithms that follow involves the following result: we can compute a matrix $Q \in \mathbb{C}^{n\times (k+p)}$, which is $B$-orthonormal, i.e. $Q^*BQ = I$ such that

\begin{equation}
\label{eqn:projectionapproximation}
\normB{(I-QQ^*B)C} \leq \varepsilon
\end{equation}
where, the range of $Q$ approximates the range of $C$, and $p$ is an oversampling factor. We define the projection matrix $P_B \define QQ^*B$ and observe that $\normB{P_B} = 1$. The reason we choose to use the $B$-norm $\normB{\cdot}$ is because $C$ is self-adjoint with respect to the B-inner product. This is implemented as follows: We draw the random matrix $\Omega$ from a standard Gaussian distribution, and form $Y=B^{-1}A\Omega$. Then we construct a matrix $Q$ that forms a basis for the range of $Y$ and is B-orthonormal. This is obtained using a QR decomposition using the $B$-inner product. The cost of computing this basis is dominated by the cost of forming the matvecs with respect to $B^{-1}$ and $A$ and computing the QR decomposition. A practical way to estimate the error in the approximation~\eqref{eqn:projectionapproximation} and the average behavior of this error $\varepsilon$ is provided in Section~\ref{sec:aposteriori}. The computational costs of this algorithm is discussed in Section~\ref{sec:compcosts}.

Several algorithms exists for QR decomposition with standard inner product $\langle x,y\rangle = y^*x$ such as Gram-Schmidt (both classical and modified), using Householder transformations and Givens rotations. However, the use of of the $B$-inner product precludes the use of Householder transformations and Givens rotations. The use of modified Gram-Schmidt for QR decomposition with weighted inner product has been discussed before (for example, see~\cite{grimes1994shifted}). It is well known that the modified Gram-Schmidt is more stable than the classical Gram-Schmidt method even for the case $B=I$. Hence, we only consider modified Gram-Schmidt approach. However, even though the computation of $R$ is extremely accurate, $Q$ is not always orthonormal (or $B$-orthonormal) due to accumulation of round-off errors. We consider two alternative algorithms: Modified Gram-Schmidt with re-orthogonalization, denoted by MGS-R, a new algorithm considered in this paper, and `PreCholQR'~\cite{lowely2014stability}.

 To ensure the B-orthogonality up to machine precision, we extend the algorithm proposed in~\cite[Section 9.3]{gander1980algorithms}, that uses the standard inner-product, to now use the $B$ inner-product. The algorithm proposed in~\cite{gander1980algorithms} was an extension to the re-orthogonalization proposed by Rutishauser. It maintains a factorization that is more accurate than MGS by accumulating changes in $R$ due to the re-orthogonalization process and unlike the standard MGS it is also designed to work even when the matrix is rank-deficient. The extension to the $B$-inner product can be accomplished readily by changing the definition of inner-products and is summarized in Algorithm~\ref{alg:mgs}. Numerical examples in Section~\ref{sec:accuracy} indicate that the modified Gram-Schmidt with re-orthogonalization is superior because it explicitly enforces orthogonality.

\begin{algorithm}[!ht]
\begin{algorithmic}[1]
\REQUIRE $Y = [y_1,\dots,y_n]$ and a positive definite matrix $W$ 
\STATE $Q := [y_1,\dots,y_n] $ and R = zeros(n,n)
\FOR {$k=1,\dots,n$} 
\STATE $\hat{q}_k = Wq_k$,  $t := \sqrt{\hat{q}_k^*q_k}$	
\STATE flag $= 1$, $ c= 0$ 
\WHILE  {flag}
\STATE$c =  c+ 1$
\FOR {$j=1,\dots,k-1$}
\STATE $s = \hat{q}_i^*q_k $,  $r_{i,k} += s$ and $q_k -= sq_i$
\ENDFOR
\STATE $\hat{q}_k := Wq_k$ and $tt = \sqrt{\hat{q}_k^*q_k}$ 
\IF {$tt > t10\epsilon$  \AND $tt < t/10$}
\STATE flag $= 1$, $t=tt$
\ELSE
\STATE flag $=0$
\IF {$tt<10\epsilon t$} \STATE $tt = 0 $ \ENDIF
\ENDIF
\STATE $r_{kk} = tt$
\IF {$tt\epsilon != 0$} \STATE $tt = 1/tt$ \ENDIF
\STATE $q_k = q_ktt$, and $\hat{q}_k = \hat{q}_ktt$
\ENDWHILE
\ENDFOR
\RETURN  $Q \in \mathbb{C}^{m\times n}$, $WQ \in \mathbb{C}^{m\times n}$ and $R \in \mathbb{C}^{n\times n}$
\end{algorithmic}
\caption{Modified Gram-Schmidt with $W$-inner products}
\label{alg:mgs}
\end{algorithm}

We also consider the `CholQR' and `PreCholQR' algorithms described and analyzed in in~\cite{lowery2014stability}. In particular, `PreCholQR' has an additional cost due to a thin QR decomposition but has better stability properties. Given a matrix $Y \in \mathbb{C}^{m\times n}$ it outputs matrices $Q$ and $R$ such that $Y=QR$ and $Q^*WQ = I$. The algorithm and the relevant matrices have been summarized in Algorithm~\ref{alg:cholqr} and~\ref{alg:precholqr}. In particular, accounting for round-off error, the resulting decompositions for PreCholQR satisfy

  \begin{align*}
\normtwo{Y-QR} \quad \leq& \quad cmn^2u\normtwo{Q}\normtwo{U}\normtwo{S} \\
\normtwo{Q^*WQ - I} \quad \leq & \quad c'mn^2u \normtwo{Q}^2\normtwo{B} + \bigO(u^2)
\end{align*}•
where $c$ and $c'$ denote constants and $u$ denotes machine precision. Numerical experiments involving the stability have been performed in Section~\ref{sec:accuracy}.

\begin{algorithm}[!ht]
\begin{algorithmic}
\REQUIRE $Y \in \mathbb{C}^{m\times n}$, $W \in \mathbb{C}^{m\times m}$ positive definite
\STATE $Z=WY$
\STATE $C = Y^*Z$
\STATE $R = \text{chol}(C)$
\STATE $Q = YR^{-1}$, $WQ = ZR^{-1}$
\RETURN $Q \in \mathbb{C}^{m\times n}$, $WQ \in \mathbb{C}^{m\times n}$ and $R \in \mathbb{C}^{n\times n}$
\end{algorithmic}
\caption{CholQR with W-inner products}
\label{alg:cholqr}
\end{algorithm}

\begin{algorithm}[!ht]
\begin{algorithmic}
\REQUIRE $Y \in \mathbb{C}^{m\times n}$, $B \in \mathbb{C}^{m\times m}$ positive definite
\STATE [Z,S] = qr(Y)
\STATE [Q,WQ,U] = CholQR(Z)
\STATE $R = US$
\RETURN $Q \in \mathbb{C}^{m\times n}$, $WQ \in \mathbb{C}^{m\times n}$ and $R \in \mathbb{C}^{n\times n}$
\end{algorithmic}
\caption{Pre-CholQR with W-inner products}
\label{alg:precholqr}
\end{algorithm}

\subsection{Two pass algorithm}

In this Section, we derive a symmetric low-rank decomposition to the GHEP in Equation~\eqref{eqn:ghep} that uses two sets of matrix-vector products involving the matrix $A$. This algorithm will be called a two pass algorithm. In Section~\ref{sec:singlepass} we will derive an algorithm that only uses one set of matrix-vector products. The single pass algorithm has a smaller computational cost but is less accurate.

Let us assume that a $Q \in \mathbb{C}^{n\times (k+p)}$ exists such that $\normB{(I-P_B)C} \leq \varepsilon$ and is relatively easy to compute. Then, we can derive the following error bound which provides the approximation error to a symmetric low-rank decomposition,

\begin{align}\label{eqn:symmapprox}
\normB{(C-P_BCP_B) } \leq & \quad \normB{C - P_BC} + \normB{P_BC - P_BCP_B} \\ \nonumber
                     \leq & \quad \varepsilon + \normB{P_B} \normB{C-CP_B} \\ \nonumber
		     \leq & \quad 2\varepsilon
\end{align}

This inequality relies on the following results $\normB{P_B} = 1$ and $ \normB{C(I-P_B)} = \normB{(I - P_B)C}$. From~\eqref{eqn:symmapprox}, we have the following low-rank decomposition
\begin{equation}
\label{eqn:symmdecomp}
C \approx P_B C P_B \quad \Rightarrow \quad A \approx (BQ) (Q^*AQ) (BQ)^*  = (BQ)T(BQ)^*
\end{equation}
where, $T\define Q^*AQ$. From this point, the eigenvalues of the system~\eqref{eqn:ghep} can be approximately computed as the eigenvalues of the matrix $T$ and the $B-$orthogonal eigenvectors $U$ can be computed by the product of $Q$ with the eigenvectors of $T$. The algorithm is summarized in Algorithm~\ref{alg:doublepass}.

Algorithm~\ref{alg:doublepass} starts by constructing a Gaussian random matrix $n\times (k+p)$ with i.i.d. entries chosen from an normal distribution with zero mean and unit variance. Here $p$ is a oversampling factor, that is chosen to lower the error in the eigenvalue calculations. Typically, $p$ is chosen to be less than $20$ following the arguments in~\cite{halko2011finding,liberty2007randomized}. The improvement in the approximation error with increasing $p$ is verified in both theory and experiment (see Sections~\ref{sec:aposteriori} and~\ref{sec:kle}). We then form matvecs with $C$ to construct $Y$. Next, we $B$-orthonormalize the columns of $Y$, using modified Gram-Schmidt with $B$-inner products. This algorithm is summarized in Algorithm~\ref{alg:mgs}. Then, we form the $(k+p)\times (k+p)$ matrix $T = Q^*AQ$, which requires a second round of matvecs with $A$. In Section~\ref{sec:singlepass}, we will describe an algorithm that avoids this second round of forming matvecs with $A$. We then compute the eigenvalue decomposition of this smaller matrix $T$, and use this to construct the approximate generalized eigendecomposition of the matrix $C$. It can be verified that $U^*BU = I$.

\begin{algorithm}[!ht]
\begin{algorithmic}[1]
\REQUIRE  matrices $A$, $B$, and $\Omega  \in \mathbb{R}^{n\times (k+p)}$  a Gaussian random matrix. Here $A,B \in \mathbb{C}^{n\times n}$, k is the desired rank, $p\sim 20$ is an oversampling factor.
\STATE  Compute $Y = C\Omega$, where $C\define B^{-1}A$
\STATE  Form QR-factorization $Y =QR$ such that $Q^*BQ = I$ 
\STATE  Form $ T \define Q^*AQ$ and
\STATE Compute the eigenvalue decomposition $T = S\Lambda S^*$. Keep the $k$ largest eigenmodes as $S = S(:,1:k)$ and $\Lambda = \Lambda(:,1:k)$. The columns of $S$ are orthonormal.
\RETURN Matrices $U \in \mathbb{C}^{n\times k}$ and $\Lambda\in \mathbb{R}^{k\times k}$
 that satisfy

\[ A  \approx B U \Lambda (BU)^*\qquad\text{with}\qquad U = QS \quad \text{and}\quad U^*BU = I\]
\end{algorithmic}
\caption{Randomized algorithm for GHEP}
\label{alg:doublepass}
\end{algorithm}

 that satisfy
%\[ A  \approx  (BU) \Lambda (BU)^*\qquad\text{with}\qquad U = QS \]
%\end{algorithmic}
%\caption{Randomized algorithm for GHEP }
%\label{alg:randomghepBinv}
%\end{algorithm}

\subsection{Single Pass algorithm}\label{sec:singlepass}
Algorithm~\ref{alg:doublepass} requires forming two sets of matvecs $Ax$ for a total of $2(k+p)$ matvecs. In some applications, matrix-vector products with $A$ can be expensive and must be used economically. It is possible to use the information already available in the matrices $Q$, $Y$ and $\Omega$ to avoid a second round of matvecs with A. This is called a single pass algorithm, following the convention in~\cite{halko2011finding}. In order to derive such an algorithm, we make the following observation. First, we define $\bar{Y} \define A\Omega$
\[ \Omega^*\bar{Y}  = \Omega^*A\Omega \approx (\Omega^*BQ)\underbrace{Q^*AQ}_{\define T}(Q^*B\Omega) \]
using the relation in~\eqref{eqn:symmdecomp}. Therefore, we can compute $T \approx  (\Omega^*BQ)^{-1}(\Omega^*\bar{Y}) (Q^*B\Omega)^{-1}$ by avoiding additional matvecs with $A$. At first glance, it appears that we need a second round of matvecs with $B$ to form $F \define Q^*B\Omega$. However, this is not the case since by using Algorithm~\ref{alg:mgs} we have both $Q$ and $BQ$. Therefore, forming $F$ only requires an additional $\bigO(k+p)^3$. We summarize the single pass algorithm in Algorithm~\ref{alg:singlepass}. Although this method is computationally advantageous, an additional error is used in computing $T$ which can be understood using Theorem~\ref{thm:singlepasserr} in section~\ref{sec:aposteriori}.
\begin{algorithm}[!ht]
\begin{algorithmic}[1]
\REQUIRE  matrices $A$, $B$ and $\Omega \in \mathbb{R}^{n\times (k+p)}$ is a Gaussian random matrix. Here $A,B \in \mathbb{R}^{n\times n}$, k is the desired rank, $p\sim 20$ is an oversampling factor.

\STATE  Compute $\bar{Y} = A\Omega$, and $Y = B^{-1}A\Omega$
\STATE  Compute $Y =QR$ such that $Q^*BQ = I$ 
\STATE  Form $ \tilde{T} = (\Omega^*BQ)^{-1}(\Omega^*\bar{Y}) (Q^*B\Omega)^{-1}$
\STATE  Compute the eigenvalue decomposition $\tilde{T} = S\Lambda S^*$. Keep the $k$ largest eigenmodes  as $S = S(:,1:k)$ and $\Lambda = \Lambda(:,1:k)$. The columns of $S$ are orthonormal.

\RETURN Matrices $U \in \mathbb{R}^{n\times k}$ and $\Lambda\in \mathbb{R}^{k\times k}$
 that satisfy

\[ A  \approx  (BU) \Lambda (BU)^*\qquad\text{with}\qquad U = QS \]
\end{algorithmic}
\caption{Randomized algorithm for GHEP - Single pass}
\label{alg:singlepass}
\end{algorithm}

We note that a different (but similar) approximation was proposed in~\cite{halko2011finding}. Starting with
\[Q^*\bar{Y} =  \Omega^*A\Omega \approx (Q^*BQ)\underbrace{Q^*AQ}_{\define T}(Q^*B\Omega) \]
where we have used  the relation in~\eqref{eqn:symmdecomp} that $A \approx (BQ)T(BQ)^*$. However, we have not pursued this approach before. %Therefore, we can compute $T \approx (Q^*\bar{Y}) (Q^*B\Omega)^{-1}$. However, we note that this approximation does not yield a symmetric decomposition (even when $B=I$). As a result, we have not pursued this approach further.

\subsection{Nystr\"{o}m method}
Yet, another alternative was proposed in~\cite{halko2011finding} to construct a low rank approximation to $A$ given a matrix $Q$ with orthonormal columns that approximates the range of $A$. The Nystr\"{o}m method builds a more sophisticated rank-$k$ approximation, namely $A \approx AQ(Q^*AQ)^{-1}Q^*A$. It can be verified that this approximation can be used without modification even for the case $B \neq I$. However, to convert this low-rank approximation $A \approx AQ(Q^*AQ)^{-1}Q^*A$ to the form $A \approx BU \Lambda (BU)^*$, we have to deviate slightly.

First, we use Cholesky factorization to factorize $T = LL^*$. Next, construct $M \define  AQL^{-*}$. Then, we use Algorithm~\ref{alg:mgs} with input matrices $Y=M$ and $W = B^{-1}$ to get $Q_MR_M = M$ such that $Q_M^*B^{-1}Q_M = I$ and $\hat{Q}_M^*B\hat{Q}_M = I$. Compute the SVD of $R_M = U_M\Sigma_MV_M^*$. Finally, we construct the low-rank factorization $A \approx BU\Lambda (BU)^*$ by constructing $U = \hat{Q}_M U_M$ and $\Lambda = \Sigma_M^2$. The algorithm is summarized in~\ref{alg:nystrom}. For numerical stability, if $T$ is rank-deficient or ill-conditioned, its inverse can be replaced with the pseudo-inverse and the algorithm proceeds similarly.

The computational cost of the Nystr\"{o}m algorithm is the same as the two-pass algorithm with an additional round of matvecs with $B^{-1}$ and an overall additional cost of $\bigO(k+p)^2n$. Theoretical and empirical results for the  Nystr\"{o}m method suggests that it is often a better approximation than the two-pass algorithm. The reason is that the Nystr\"{o}m method is essentially performing (for free) an additional step of the randomized power iteration described in~\cite[Algorithm 4.3]{halko2011finding}.

\begin{algorithm}[!ht]
\begin{algorithmic}[1]
\REQUIRE  matrices $A$, $B$ and $\Omega \in \mathbb{R}^{n\times (k+p)}$ is a Gaussian random matrix. Here $A,B \in \mathbb{R}^{n\times n}$, k is the desired rank, $p\sim 20$ is an oversampling factor.

\STATE  Compute $Y = B^{-1}A\Omega$
\STATE  Compute $Y =QR$ such that $Q^*BQ = I$ using modified Gram-Schmidt (see Algorithm~\ref{alg:mgs}).
\STATE  Form $ T= Q^*AQ$ and compute the Cholesky factorization $T = LL^*$
\STATE  Form $M = AQL^{-*}$
\STATE  Using Algorithm~\ref{alg:mgs} with $W=B^{-1}$ to get $Q_MR_M = M$ such that $Q_M^*B^{-1}Q_M = I$ and $\hat{Q}_M^*B\hat{Q}_M = I$.
\STATE Compute the SVD of $R_M = U_M\Sigma_MV_M^*$. Keep the $k$ largest modes  as $U_M = U_M(:,1:k)$ and $\Lambda = \Sigma_M(:,1:k)^2$.

\RETURN Matrices $U \in \mathbb{R}^{n\times k}$ and $\Lambda\in \mathbb{R}^{k\times k}$  that satisfy

\[ A  \approx  (BU) \Lambda (BU)^*\qquad\text{with}\qquad U = \hat{Q}_FU_F\]
\end{algorithmic}
\caption{Randomized algorithm for GHEP - Nystr\"{o}m version}
\label{alg:nystrom}
\end{algorithm}

\subsection{Summary of computational costs}\label{sec:compcosts}
We now briefly discuss the costs associated with the various algorithms described so far. The cost of the two pass algorithm is $2(k+p)$ matvecs with $A$, $(k+p)$ matvecs and $B^{-1}x$ and an additional $\bigO (k+p)^2n$ operations for forming the approximate eigenvalues and eigenvectors. The $B$-orthogonalization is accomplished using Algorithm~\ref{alg:mgs} which only uses one set of $(k+p)$ matvecs with $B$ (assuming no re-orthogonalization), but in return we get two sets of vectors $Q$ and $\hat{Q}$ which satisfy $Q^*BQ = I$ and $\hat{Q}^*B^{-1}\hat{Q} = I$. The Modified Gram-Schmidt also requires $\bigO(k+p)^2n$ operations for computing inner products. The single pass algorithm, on the other hand, only uses one set of matvecs with A. The comparison of the costs between the algorithms is summarized in Table~\ref{tab:compcosts}. However, it should be noted that if re-orthogonalization occurs in the modified Gram-Schmidt algorithm then the number of matvecs involving $B$ and $B^{-1}$ could be higher.

 However, under certain circumstances, the algorithms described can be further accelerated. We provide a few examples:
\begin{itemize}
\item It is sometimes advantageous to apply a matrix to $k+p$ vectors simultaneously rather than execute $k+p$ matvecs consecutively. For example, out-of-core finite-element codes are more efficient when they are programmed to exploit the presence of a block of the matrix $A$ in fast memory, as much as possible~\cite{saad1992numerical}.
\item Computing $A\Omega$ and $B^{-1}A\Omega$ can be trivially parallelized. Since this is often the chief bottleneck, considerable gains might be obtained by parallelism.
%\item Applying the operator $B^{-1}A$ to a vector is sometimes more efficient than applying $Ax$ and then computing $B^{-1}(Ax)$. An example of this situation is discussed in section~\ref{sec:kle}.
%\item In certain cases, applying $B^{-1}$ on a vector is cheaper to form than $B$ on a vector. Consider the transformed problem $B^{-1}AB^{-1}y = \lambda B^{-1}y$ which has the same eigenvalues as $Ax = \lambda Bx$ with eigenvectors $x = B^{-1}y$. Algorithms~\ref{alg:doublepass} and~\ref{alg:singlepass} can be applied to the transformed pair $(B^{-1}AB^{-1},B^{-1})$ with $B^{-1}$ inner products. This results in a fast algorithm that avoids products with $B$.
 \end{itemize}
It should be noted that the gains from using randomized techniques in comparison to classical methods (such as Krylov subspace methods) is not because they have a smaller computational cost but rather because they allow us to to reorganize our calculations such that we can fully exploit matrix properties and the computer architecture~\cite{halko2011finding}. %In particular, the last two examples mentioned above can be exploited using Krylov subspace methods as well.

\begin{table}[!ht]
\centering
\begin{tabular}{|c|c|c|c|c|}\hline
Method \ Cost & $Ax$ & $Bx$ & $B^{-1}x$ & Scalar work\\ \hline
Two Pass Algorithm~\ref{alg:doublepass} & $2(k+p)$ & $(k+p)$ & $(k+p)$ &$\bigO(k+p)^2n$   \\ \hline
Single pass Algorithm~\ref{alg:singlepass} & $(k+p)$ & $(k+p)$ & $(k+p)$ & $\bigO(k+p)^2n$  \\ \hline
Nystr\"{o}m Algorithm~\ref{alg:nystrom} & $2(k+p)$ & $(k+p)$ & $2(k+p)$ &$\bigO(k+p)^2n$\\ \hline
\end{tabular}
\caption{Summary of computational costs (assuming no re-orthogonalization in the modified Gram-Schmidt algorithm)}
\label{tab:compcosts}
\end{table}

\section{Generalized Singular Value Decomposition}
Generalized SVD (GSVD) is often used in the context of inverse problems and deblurring. It has applicability both as an analytical tool and practical utility in computing minimum norm solutions in regularized weighted least squares problems. Two different generalizations of the SVD have been discussed in~\cite{van1976generalizing}. Here we consider the second definition in~\cite[definition 3]{van1976generalizing}, with $A\in \mathbb{C}^{m\times n}$ and two positive definite matrices $S\in\mathbb{C}^{m\times m}$ and $T\in\mathbb{C}^{n\times n}$
\begin{equation}\label{eqn:gensingular} \sigma_{S,T} (A) = \left\{\mu \left| \mu \text{ are the stationary points of } \frac{\norm{Mx}{S}}{\norm{x}{T}} \right.\right\}\end{equation}
This results in a following decomposition of the form
\[ U^{-1}AV = \Sigma_{S,T}   \qquad U^*SU=I \qquad  V^*TV=I\]
and $\Sigma_{S,T} = \text{diag}\{\sigma_{S,T,1},\dots,\sigma_{S,T,n}\}$ are the generalized singular values.

A simple modification of the algorithms for GHEP yields us an algorithm for the GSVD as defined above.  We first compute $Y_1=A\Omega_1$ and $Y_2 = A^*\Omega_2$. We S-orthonormalize $Y_1$ and T-orthonormalize $Y_2$ using Algorithm~\ref{alg:precholqr} so that $Y_1 = Q_1R_1$ with $Q_1^*SQ_1 = I$ and $Y_2 = Q_2R_2$ with $Q_2^*TQ_2 = I$. We have the following error bounds $\normtwo{(I-Q_1Q_1^*S)A} \leq \varepsilon_S$ and $\normtwo{(I-Q_2Q_2^*T)A^*} \leq \varepsilon_T$. Error bounds of the type derived in Proposition~\ref{prop:random} can be established in this case as well. It can be shown that
\[ \normtwo{A-Q_1Q_1^*SATQ_2Q_2^*} \leq \varepsilon_S +  \varepsilon_T\normtwo{Q_1Q_1^*S}\]
Based on the above approximate low-rank representation, compute $F = Q_1^*SATQ_2$ and compute its SVD $F= \tilde{U}\Sigma \tilde{V}^*$. Then, the approximate GSVD can be computed using
\[ A \approx U \Sigma V^*\qquad U = Q_1\tilde{U} \qquad V = Q_2\tilde{V}\]
 and the matrices $U$ and $V$ satisfy the relations $U^*SU = I$ and $V^*TV = I$.

Generalized SVD is more popularly defined in the following form~\cite{van1976generalizing,golub2012matrix}: given two matrices $A\in \mathbb{C}^{m_A\times n}$ and $B \in \mathbb{C}^{m_B\times n}$, with $m_A \leq n, $the GSVD is given by
\[ A = UCX^* \quad B = VSX^*\]
where, $U \in \mathbb{C}^{m_A\times m_A}$ and $V \in \mathbb{C}^{m_B\times m_B}$ are unitary matrices, $X\in \mathbb{C}^{n\times n}$ is a square matrix $C,S$ are diagonal matrices with non-negative entries and satisfy the relation $C^*C + S^*S = I$. The generalized singular values are given by $\sigma(A,B)$ are given by the ratio of the diagonal entries of $C$ and $S$. The relation between the two definitions presented here is that when rank$(B)=n$, the generalized singular values of the matrix pair $\sigma(A,B) = \sigma_{S,T}$ with $S= I_{m_A}$ and $T= B^*B$.

\section{Convergence and a posteriori error bounds}\label{sec:aposteriori}

The idea of randomized algorithms is to compute matrix-vector products involving matrix $C \define B^{-1}A$ with vectors $\omega_i$ that have i.i.d. entries chosen from standard normal distribution. These columns, when appropriately orthonormalized form an approximate basis for the column space spanned by the eigenvectors corresponding to the largest eigenvalues. In Section~\ref{sec:randomized}, we assumed that we can compute a $Q\in \mathbb{R}^{n\times (k+p)}$ that satisfied the error bound~\eqref{eqn:projectionapproximation}. In order to estimate the resulting error in the low-rank representation $\varepsilon$, we use the following result stated in the form of a proposition.
\begin{propos}\label{prop:random}
Draw a sequence of random vectors $\omega_i$ that have i.i.d. entries chosen from standard normal distribution. Let $C \define B^{-1}A$ with $A$ symmetric and $B$ symmetric positive definite. Fix a positive integer $r$ and $\alpha > 1 $.
\begin{equation}
\label{eqn:randomapprox}
\normB{(I-QQ^*B)C}  \leq  \alpha \sqrt{\frac{2\normtwo{B^{-1}}}{\pi}} \smash{\displaystyle\max_{i = 1,\dots,r}} \normB{(I-QQ^*B)C\omega_i}
\end{equation}
holds with probability at least $1-\alpha^{-r}$.
\end{propos}

\begin{proof}

Using the relation in Equation~\eqref{eqn:prop2}  we have the inequality that
\[\normB{(I-QQ^*B)C}  =  \normtwo{B^{1/2}(I-QQ^*B)CB^{-1/2} }\leq \sqrt{\normtwo{B^{-1}}}\normtwo{B^{1/2}(I-QQ^*B)C }\] Define the matrix $M = B^{1/2}(I-QQ^*B)C$ and using the result from~\cite[lemma 4.1]{halko2011finding} to the matrix $M$, we arrive at
\begin{align*}
\normB{(I-QQ^*B)C} \quad  \leq& \quad   \alpha \sqrt{\frac{2\normtwo{B^{-1}}}{\pi}}  \smash{\displaystyle\max_{i = 1,\dots,r}} \normtwo{B^{1/2}(I-QQ^*B)C\omega_i} \\
= & \quad \alpha \sqrt{\frac{2\normtwo{B^{-1}}}{\pi}}  \smash{\displaystyle\max_{i = 1,\dots,r}} \normB{(I-QQ^*B)C\omega_i}
\end{align*}
 holds with probability at least $1-\alpha^{-r}$.
\end{proof}

In practice, $\normtwo{B^{-1}}$ might not be easy to compute. Instead, we propose a crude estimator that is easy to compute. Observe that $\normB{q_i} = 1$. Using inequality~\eqref{eqn:Bnorm2normineq}, we have
\begin{equation}\label{Binvnormapprox}
\frac{\normtwo{q_i}^2}{\normtwo{B^{-1}}} \leq \normB{q_i}^2 = 1 \quad \Rightarrow \quad \sqrt{\normtwo{B^{-1}}} \geq \max_{i=1,\dots,r} \normtwo{q_i}
\end{equation}
The significance of the Proposition~\ref{prop:random} is that we now have an easy to compute a-posteriori bound for our error that can be obtained by forming matvecs with C. However, as~\cite{halko2011finding} suggests, this is a crude estimate. The cost of this estimator is mostly performing matvecs with $A$ and $B^{-1}$. Thus, we can make a guess for the numerical rank of $B^{-1}A$, compute the low-rank approximation $C \approx QQ^*BC$, evaluate the error estimate in Proposition~\ref{prop:random} and keep adding more samples if this error estimate is too large. However, the matvecs $B^{-1}A\Omega$ performed on random vectors for the error estimator can be re-used. As a result, the error estimator is almost free of cost.  A better estimate can be obtained by using power iteration acting on a random vector.

The analysis in~\cite{halko2011finding} suggests that if the spectrum of $C$ decays rapidly, then the error in the approximation is quite small. We are now ready to state our main result and defer the proof to the Appendix.

\begin{theorem}\label{thm:main}
Let $Q$ be computed according to Algorithm~\ref{alg:doublepass} by choosing a Gaussian random matrix $\Omega \in \mathbb{R}^{n\times r}$ with $r=k+p$. Let $C=U\Sigma_B V^*$ be the singular value decomposition in the generalized sense~\eqref{eqn:gensingular}. We have the inequality
\[ E\normB{(I-P_B)C} \leq\sqrt{\normtwo{B^{-1}}}\left[ \left( 1+\sqrt{\frac{k}{p-1}}\right)\sigma_{B,k+1} +\frac{e\sqrt{k+p}}{p}\left(\sum_{j=k+1}^n\sigma_{B,j}^2\right)^{1/2} \right]  \]
where, $\sigma_{B,j}$ for $j=1,\dots,n$ are the generalized singular values given by~\eqref{eqn:gensingular} and $E[\cdot]$ denotes the expectation.
\end{theorem}

In addition to the average spectral error, an expression for the deviation bounds of the spectral error can be derived similar to~\cite[Theorem 10.8]{halko2011finding}. The spectral error suggests that if the singular values (in the generalized sense) are decaying rapidly then the error due to the low-rank approximation is small, in expectation. Based on the analysis in~\cite{halko2011finding}, this result is not surprising and we defer the proof of this theorem to the appendix. The generalized singular value decomposition can be computed using the algorithm described in~\cite{van1976generalizing}. However, this approach requires forming square roots of $B$. We can instead use the following inequality to provide an estimate $\sigma_{B,k} \leq \sqrt{\normtwo{B}} \sigma_{k}$. Furthermore, the above error bound suggests that the error is high when $\normtwo{B^{-1}}$ is large. In several cases $B^{-1}$ is bounded, for instance in finite elements, where $B$ is the mass matrix and is spectrally equivalent to the identity operator. Otherwise, if some combination $(\alpha A + \beta B)$ can be found such that its inverse has a small norm, we can instead solve the transformed problem $ Ax  = \theta (\alpha A + \beta B)x $.

 At this point, we have provided both an a-priori and a posteriori measure of error in the low-rank approximation. However, does a small error in the low-rank approximations imply that there is a small error in the subsequent eigenvalue calculations? In order to answer this question, we turn to some results from the theory of spectral approximation. We now derive expressions for the error between the computed eigenvalues and the true eigenvalues and the angle between the true and approximate eigenvectors. It should be noted that a result of this kind is common in the theory of perturbation for eigenvalues of Hermitian matrix and makes use of the Kato-Temple Theorem~\cite[Theorem 3.8]{saad1992numerical} and~\cite[Section 7.1, chapter 5]{bai1987templates}.

\begin{propos}\label{prop:apost}
Let $Q$ satisfy the relation $\normB{(I-P_B)C} \leq \varepsilon$ so that $\normB{C-P_BCP_B} \leq 2\varepsilon$. Let the eigenpair $(\tilde{\lambda},\tilde{u})$ be an approximation to the  eigenvalue problem $Ax = \lambda Bx $ calculated by Algorithms~\ref{alg:doublepass,alg:singlepass,alg:nystrom}. Then we have the following error bounds
\[ |\lambda-\tilde{\lambda}| \leq \min\{2\varepsilon,\frac{4\varepsilon^2}{\delta}\}\qquad \sin\angle_B (u,\tilde{u}) \leq \frac{2\varepsilon}{\delta}\]

where, $\delta = \min_{\lambda_i\neq\lambda} |\tilde{\lambda}-\lambda_i|$ is the gap between the approximate eigenvalue $\tilde{\lambda}$ and any other eigenvalue and $\angle_B(x,y) = \arccos \frac{|<x,y>_B|}{\normB{x}\normB{y}}$
\end{propos}

\begin{proof}
We start by defining the residual corresponding to the approximate eigenpair $r = A\tilde{u}-\tilde{\lambda}B\tilde{u}$. We first start with the proof that $\norm{r}{B^{-1}} \leq {2\varepsilon} $. By definition, $\norm{r}{B^{-1}}  = \normtwo{B^{-1/2}r}$. Plugging in the expression for $r$, we have

 \[\normtwo{B^{-1/2}r} = \normtwo{B^{1/2}(C\tilde{u}- \tilde{\lambda}\tilde{u})} = \normB{C\tilde{u}-\tilde{\lambda}\tilde{u}} \]
   Also, in a slight change of notation from Algorithm~\ref{alg:doublepass}, we denote the approximate eigenpairs by $(\tilde{\lambda}_i,\tilde{u}_i)$ for $i=1,\dots,r$ to distinguish it from the exact eigenpair $(\lambda,u)$. We have $T=Q^*AQ = S\tilde{\Lambda}S^*$ and $\tilde{U} = QS$. We make the following observations: 1) $P_BCP_B = QQ^*AQQ^*B = \tilde{U} \tilde{\Lambda} \tilde{U}^*B $ and 2) since $\tilde{u}$ is a column of the B-orthonormal matrix $\tilde{U}$, we have $\tilde{\lambda}\tilde{u} = \tilde{U} \tilde{\Lambda} \tilde{U}^*B \tilde{u}$

   Using these observations,
\[ \normB{C\tilde{u}-\tilde{\lambda}\tilde{u}} = \normB{C\tilde{u}-\tilde{U}\tilde{\Lambda}\tilde{U}^*B\tilde{u}} \leq \normB{C-\tilde{U}\tilde{\Lambda}\tilde{U}^*B} = \normB{C-P_BCP_B} \leq 2\varepsilon\]

Then from~\cite[Section 7.1, chapter 5]{bai1987templates}, we have the following relations
\[ |\lambda-\tilde{\lambda}| \leq \norm{r}{B^{-1}} \qquad|\lambda-\tilde{\lambda}| \leq \frac{\norm{r}{B^{-1}}^2}{\delta} \qquad \sin\angle_B (u,\tilde{u}) \leq \frac{\norm{r}{B^{-1}}}{\delta}\]
The proof is completed by plugging in the inequality $\norm{r}{B^{-1}} \leq 2\varepsilon$.
\end{proof}

%Older sentence - improved for reviewer 1
%While considering the accuracy, it is not only the error in low-rank representation that controls the error in the eigenvalues calculations - the proposition suggests an additional parameter that controls the accuracy of the resulting eigenvalue calculations is the spectral gap $\delta$, i.e. the gap between the approximate eigenvalue $\tilde{\lambda}$ and any other eigenvalue.

The error in low-rank representation is not the only factor that controls the error in the eigenvalues calculations. Proposition suggests that the accuracy is also determined by an additional parameter called the \textit{spectral gap} $\delta$, defined as the gap between the approximate eigenvalue $\tilde{\lambda}$ and any other eigenvalue. When the eigenvalues are clustered, the spectral gap is small and the eigenvalue calculations are accurate as long as the error in the low-rank representation is small. However, in this case the resulting eigenvector calculations maybe inaccurate because the parameter $\delta$ appears in the denominator for the approximation of the angle between the true and approximate eigenvector.

The following result provides an upper bound for the difference in the eigenvalues computed using the two-pass and single-pass algorithms as described in Algorithm~\ref{alg:doublepass} and Algorithm~\ref{alg:singlepass} respectively. Numerical results confirm that typically, the two-pass algorithm is more accurate than the single pass algorithm.

\begin{theorem}\label{thm:singlepasserr}
Let $\tilde{T}$ be computed using the expression $\tilde{T} = (\Omega^*BQ)^{-1}(\Omega^*\bar{Y})(Q^*B\Omega)^{-1}$. Furthermore, assume that $Q$ satisfies the error bound $\normB{(I-QQ^*B)C} \leq \varepsilon$. Label the  eigenvalues of $T = Q^*AQ$ as $\mu_1,\dots,\mu_{k+p}$ and the eigenvalues of $\tilde{T}$ as $\theta_1,\dots,\theta_{k+p}$. The eigenvalues $\mu_j$ and $\theta_j$ for $j=1,\dots,k+p$ are related by the inequality
\[ |\mu_j-\theta_j | \leq 2\varepsilon\sqrt{\kappa(B)} \frac{\sigma_\text{max}^2 (\Omega)}{\sigma_\text{min}^2 (F)} \]
where, $F \define Q^*B\Omega$ and $\kappa(B) = \normtwo{B}\normtwo{B^{-1}}$ is the condition number of the matrix B.
\end{theorem}
\begin{proof}
We start with bounding the error  $\normtwo{T-\tilde{T}}$, where $T=Q^*AQ$.
%\[ \normtwo{T-\tilde{T}} = \normtwo{F^{-T}F^T T FF^{-1} -F^{-T}(\Omega^*A\Omega )F^{-1} } \leq \normtwo{F}^2\normtwo{\Omega BQQ^*AQQ^*B\Omega  - \Omega^*A\Omega} \leq 2\varepsilon \normtwo{\Omega}^2\normtwo{F}^2 \]
\begin{align*}
\normtwo{T-\tilde{T}} \quad = & \quad \normtwo{F^{-*}F^* T FF^{-1} -F^{-*}(\Omega^*A\Omega )F^{-1} }  \\ \nonumber
		= & \quad \normtwo{ F^{-*}\Omega^*BQ(Q^*AQ)Q^*B\Omega F^{-1}- F^{-*}\Omega^*A\Omega F^{-1} } \\ \nonumber
		\leq & \quad \normtwo{A - BQ(Q^*AQ)(BQ)^*}\normtwo{\Omega F^{-1}}^2
\end{align*}
From the the assumption that $\normB{(I-QQ^*B)C} \leq \varepsilon$ and Equation~\eqref{eqn:symmapprox} and  we have that $\normB{A-BQ(Q^*AQ)(BQ)^*} \leq 2\varepsilon$. For a matrix $M$, it can be shown that
\[  \frac{\normtwo{M}}{\sqrt{\kappa(B)}}\quad  \leq \quad  \normB{M} \quad \leq \quad  \sqrt{\kappa(B)} \normtwo{M}\]
As a result, $\normtwo{A-BQ(Q^*AQ)(BQ)^*} \leq 2\varepsilon\sqrt{\kappa(B)} $. Finally, putting it all together,
\[ \normtwo{T-\tilde{T}} \quad \leq\quad  2\varepsilon \sqrt{\kappa(B)} \frac{\sigma_\text{max}^2 (\Omega)}{\sigma_\text{min}^2 (F)} \]
 Finally, applying the Bauer-Fike Theorem~\cite[Theorem 3.6]{saad1992numerical}, and using the fact that matrix $T$ is symmetric and has an orthonormal eigenvectors, we have the desired result.
\end{proof}

The error bound in Theorem~\ref{thm:singlepasserr} provides insight into the error made using the single pass approximation. As a consequence, it is important to understand that the error in the single pass approximation can significantly degrade the approximation of the eigenvalues. The terms that contribute are 1) error in the low-rank decomposition $\varepsilon$, 2) ill-conditioned matrices $B$, and 3) large $\sigma_\text{max} (\Omega)$ and small $\sigma_\text{min} (F)$. The largest singular value of $\Omega$ is asymptotically $\sqrt{n}$ for $k  \ll n$~\cite{halko2011finding}, so the single pass approximation is poor when the sizes of the matrices are large. %Finally, $\sigma_\text{min} (F)$ is small when $F$ is ill-conditioned.

\section{Karhunen-Lo\`{e}ve expansion}\label{sec:kle}

\subsection{Motivation and background}
The Karhunen-Lo\`{e}ve expansion (KLE)~\cite{ghanem1991stochastic} is a representation of a stochastic process as an infinite linear combination of orthogonal functions, analogous to a Fourier series representation of a function. In contrast to a Fourier series where the coefficients are real numbers and the expansion basis consists of sinusoidal functions, the coefficients in the Karhunen-Lo\`{e}ve Theorem are random variables and the expansion basis depends on the process. In fact, the orthogonal basis functions used in this representation are determined by the covariance function of the process.  The random field is characterized by a mean and a covariance function. The KLE requires the computation of eigenpairs, which are derived from an Fredholm integral eigenvalue problem with the covariance function as the kernel. Consider the random field $s(\textbf{x})$, with mean $\mu(\textbf{x})$ and covariance $\kappa(\textbf{x},\textbf{y})$, on the bounded domain $\textbf{x}\in{\cal{D}}$.  The covariance kernel is assumed to be symmetric and positive definite. The KLE can now be written as
\begin{equation}
 \label{eqn:kle}
s(\textbf{x}) = \mu(\textbf{x}) + \sum_{i=1}^\infty \xi_i\sqrt{\lambda_i}\phi_i(\textbf{x})\quad \text{with,}
\end{equation}

\[ \mu(\bx)=E[s(\textbf{x})], \qquad \xi_i = \frac{1}{\sqrt{\lambda_i}}\int_{ {\cal{D}}} (s(\textbf{x})-\mu(\textbf{x}) )\phi_i(\textbf{x}) d\bx\]
Here, $\xi_i$ are uncorrelated random variables, $(\lambda_i,\phi_i(\textbf{x}))$ are the eigenpair obtained as the solution to the Fredholm integral equation of the second kind
\begin{equation}
 \label{eqn:intkle}
\int_{\cal{D}} \kappa(\textbf{x},\textbf{y})\phi(\textbf{y})d\textbf{y} = \lambda \phi(\textbf{x})
\end{equation}
Since the covariance $\kappa(\cdot,\cdot)$ is symmetric and positive definite, the eigenfunctions $\phi_i(\cdot)$ are mutually orthogonal and form a basis for $L^2({\cal{D}})$ and the eigenvalues $\lambda_i$ are real, non-negative and can be arranged in decreasing order $\lambda_1\geq\lambda_2\geq \dots \geq 0$. If the random field is Gaussian, then $\xi_i \sim{\cal{N}}(0,1)$. Further details are provided in~\cite{ghanem2003stochastic}.

The eigenpair $(\lambda_i,\phi_i(\textbf{x}))$ in the KLE, can be computed by first discretizing the weak form of system of Equations~\eqref{eqn:intkle} (i.e. performing a Galerkin projection) using piecewise linear basis functions and, subsequently solving the linear eigensystem using a generalized eigenvalue solver for symmetric matrices, that requires only matrix-vector products involving the discretized operator. The relevant equations after discretization are
\begin{equation}
\label{eqn:klediscrete}
M\prior M \phi_i = \lambda_i M \phi_i \qquad i = 1,\dots,N
\end{equation}
where, $\prior$ is the covariance matrix that arises from the discrete representation of the Gaussian random field corresponding to the covariance kernel $\kappa(\cdot,\cdot)$. $M$ is the mass matrix $M_{ij} = \int_{\cal{D}} v_i v_j d\textbf{x}$ and $v_i$ are the piecewise linear basis functions. The mass matrix is a discrete representation of the continuous identity operator and hence we expect it to be well-conditioned. We define $A\define M\prior M$ and $B\define M$. We also have that $B^{-1}A = \prior M$ is symmetric with respect to $M$-inner products. The KLE is truncated to a finite number terms $K$, which typically far fewer than the number of basis functions and independent of it. The number of terms retained in the series depends on the decay of the eigenvalues, which, in turn depends on the smoothness of the covariance kernel~\cite{schwab2006karhunen}. When the kernel is piecewise smooth, then the eigenvalues decay algebraically, and when the kernel is piecewise analytic then the decay is exponential. This GHEP nicely fits the requirements of the randomized algorithm, since it has rapidly decaying eigenvalues.

\subsection{Accuracy of the eigenvalue calculations}\label{sec:accuracy}
We consider three different covariance kernels chosen from the Mat\'{e}rn covariance family with $d = \normtwo{\bx-\textbf{y}}/l$
\begin{equation}
\label{eqn:matern3}
\kappa_\nu (\bx,\textbf{y})  = \left\{ \begin{array}{ll} \exp(-d)  & \quad \nu = 1/2\\ (1+\sqrt{3}d)\exp(-\sqrt{3}d) & \quad \nu = 3/2 \\ (1+\sqrt{5}d + \frac{5}{3}d^2)\exp(-\sqrt{5}d) & \quad \nu = 5/2 \end{array} \right.
\end{equation}
In the rest of this subsection, we consider the KLE corresponding the covariance kernels defined in Equation~\eqref{eqn:matern3} defined on the domain $x \in [-1,1]$ with the length scale parameter chosen to be $l=2$. The domain has been discretized using $201$ grid points. We deliberately chose a small problem to compare the accuracy against the results obtained from direct algorithms. We note that the rate of decay of eigenvalues is higher for covariance kernels with increasing values of $\nu$, thus providing a wide range of eigenvalue decays to study the performance of our algorithm.
\begin{table}[!ht]\centering
 \begin{tabular}{|c|c|c|c|c|}\hline
Kernel & $\normtwo{QR-Y}$ & $\normtwo{Q^*BQ - I}$ & $\normtwo{Q^*BY - R}$ & $\normtwo{YR^{-1}-Q}$ \\ \hline
\multicolumn{5}{|c|}{Modified Gram-Schmidt} \\ \hline
$\kappa_{1/2}(r) $ & $1.8\times 10^{-15}$ & $1.1\times 10^{-11}$ & $1.7\times 10^{-11}$ & $5.5\times 10^{-11}$ \\
$\kappa_{3/2}(r) $ &
$2.3 \times 10^{-15}$ &
$1.3 \times 10^{-7}$ &
$2.4 \times 10^{-7}$ &
$8.8 \times 10^{-7}$ \\
$\kappa_{5/2}(r) $ &
$2.2 \times 10^{-15}$ &
$6.1  \times 10^{-4}$ &
$1.2  \times 10^{-3}$ &
$5.4  \times 10^{-3}$  \\ \hline
\multicolumn{5}{|c|}{Modified Gram-Schmidt with re-orthogonalization} \\ \hline
$\kappa_{1/2}(r) $ &
$1.7 \times 10^{-15}$ &
$1.5 \times 10^{-15}$ &
$1.5 \times 10^{-15}$ &
$5.8 \times 10^{-11}$ \\
$\kappa_{3/2}(r) $ &
$2.1 \times 10^{-15}$ &
$1.1 \times 10^{-15}$ &
$1.0 \times 10^{-15}$ &
$8.2 \times 10^{-7}$ \\
$\kappa_{5/2}(r) $ &
$2.3 \times 10^{-15}$ &
$1.7 \times 10^{-15}$ &
$1.0 \times 10^{-15}$ &
$5.6 \times 10^{-3}$ \\ \hline
\multicolumn{5}{|c|}{PreCholQR} \\ \hline
$\kappa_{1/2}(r)$ &
$1.06\times 10^{-14}$ &
$1.17\times 10^{-15}$ &
$9.84\times 10^{-16}$ &
$1.43\times 10^{-10}$ \\ 
$\kappa_{3/2}(r)$ &
$9.06\times 10^{-15}$ &
$1.11\times 10^{-15}$&
$ 7.01 \times 10^{-16}$ &
$2.79\times 10^{-06}$ \\ 
$\kappa_{5/2}(r) $ &
$9.78 \times 10^{-15}$ & 
$1.15\times 10^{-15}$ &
$8.78\times 10^{-16}$ & 
 $2.8\times 10^{-2}$ \\ \hline
\end{tabular}
\caption{Comparison of the algorithms for computing the QR decomposition of $Y = B^{-1}A\Omega$ where $\Omega \in \mathbb{R}^{201\times 100}$ with i.i.d. entries chosen from ${\cal{N}} (0,1)$. Further, $A=M\prior M$ and $B=M$.   }
\label{tab:mgs}
\end{table}

\subsubsection{Accuracy of QR with weighted inner product:}  We compare the algorithms for computing the QR decomposition of $Y = B^{-1}A\Omega$ where $\Omega \in \mathbb{R}^{201\times 100}$ with i.i.d. entries chosen from ${\cal{N}} (0,1)$. For the modified Gram-Schmidt algorithm (MGS), we consider the algorithm in~\cite{grimes1994shifted} without additional re-orthogonalization. For the algorithm with re-orthogonalization (MGS-R) we consider the one proposed in Algorithm~\ref{alg:mgs}. We compare the following metrics: $\normtwo{QR-Y}$,  $\normtwo{Q^*BQ - I}$, $\normtwo{Q^*BY - R}$  and $\normtwo{YR^{-1}-Q}$. If the quantities were computed in exact arithmetic, they would all be identically zero. However, in the presence of round-off errors, these quantities are not numerically zero. We compare the results for three different covariance kernels defined in~\eqref{eqn:matern3} and we have $A=MQM$ and $B=M$. The results are summarized in Table~\ref{tab:mgs}. We clearly see that as $\nu$ increases the eigenvalues of the KLE decay rapidly, as a result $Y$ becomes more and more ill-conditioned. Applying the algorithm MGS results in the quantity $\normtwo{QR-Y}$ being satisfied to nearly machine precision. However, the other metrics  $\normtwo{Q^*BQ - I}$, $\normtwo{Q^*BY - R}$  and $\normtwo{YR^{-1}-Q}$ perform badly as $\nu$ increases. On the other hand, for the re-orthogonalized MGS (MGS-R) the quantities $\normtwo{QR-Y}$,  $\normtwo{Q^*BQ - I}$ and $\normtwo{Q^*BQ - R}$ are satisfied to nearly machine precision. However, like MGS $\normtwo{YR^{-1}-Q}$ is higher because $R$ is close to singular. It is clear that  while re-orthogonalization has a significant effect on the orthogonality of $Q$, it comes at a higher expense because of additional re-orthogonalization. The accuracy of `PreCholQR' is comparable with MGS-R. Unless mentioned explicitly we use MGS-R throughout this section for all the numerical experiments.

\begin{figure}[!ht]
\centering
\includegraphics[scale=0.45]{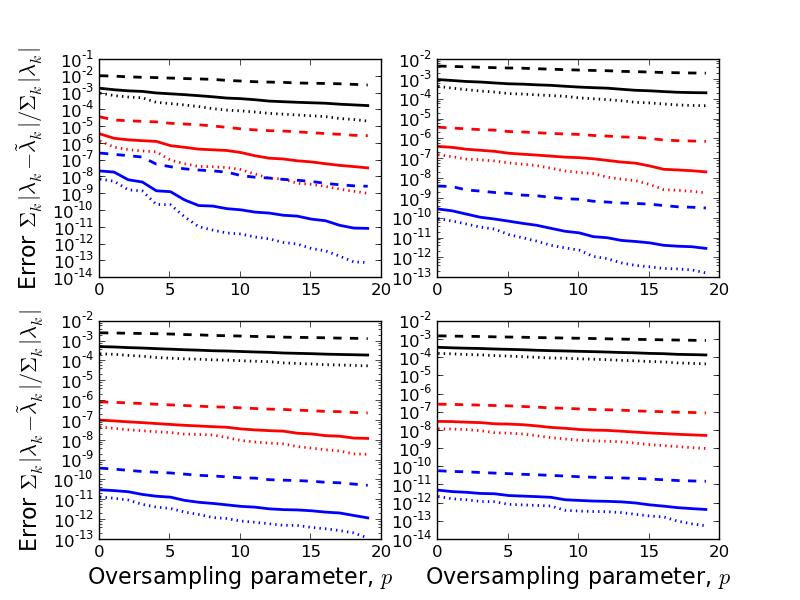}
\caption{Comparison of the error between the true eigenvalues $\lambda_k$ and the approximate eigenvalues $\tilde\lambda_k$ as a function of oversampling parameter $p$ for each of the covariance kernels defined in Equation~\eqref{eqn:matern3} - $\kappa_{1/2}$ (black), $\kappa_{3/2}$ (red) and $\kappa_{5/2}$ (blue). The plots correspond to $k=20,40,60$ and $80$. 2-pass algorithm (solid line) refers to Algorithm~\ref{alg:doublepass},  1-pass (dashed line) algorithm refers to Algorithm~\ref{alg:singlepass} and Nystr\"{o}m algorithm (dotted line) refers to Algorithm~\ref{alg:nystrom}.
 }
\label{fig:errcovker}
\end{figure}

\subsubsection{Effect of oversampling parameter $p$:} We consider the effect of the oversampling parameter $p$ on the accuracy of the low-rank approximation and the computed eigenvalues. We plot (in Figure~\ref{fig:errcovker}) the error using two-pass, single-pass  and Nystr\"{o}m algorithms applied to all three covariance kernels defined in Equation~\ref{eqn:matern3} as a function of the oversampling parameter $p$. For fairness in comparison, to eliminate the effect of random sampling, we use the same sequence of pseudo-random numbers while computing the low-rank decomposition. We can see from Theorem~\ref{thm:main} that by increasing $p$, the error of the low-rank estimate improves. However, the rate of improvement of the error with increasing oversampling also seems to increase when the rate of decay of the singular values is higher. This is consistent with the result of Theorem~\ref{thm:main}. However, while the error decreases while using the two-pass and the single-pass algorithms, the rate of improvement of error with increased oversampling is more pronounced in the case of two-pass and Nystr\"{o}m algorithms. This is because, in the single-pass algorithm, an additional error is introduced while converting the low-rank decomposition $A \approx (BQ)(Q^*AQ)(BQ)^*$ to a generalized eigendecomposition of the form $A \approx (BU)\Lambda(BU)^*$. To gain more insight, we consider the error between the matrices $T$ (that is formed exactly in the two-pass and Nystr\"{o}m algorithms) and its approximation $\tilde{T}$ (that is formed in the single-pass algorithm) as a function of oversampling parameter $p$ for each of the covariance kernels. The results are displayed in Figure~\ref{fig:errsinglepass}. The error between $T$ and $\tilde{T}$ decreases with oversampling although slowly.

\begin{figure}[!ht]
\centering
\includegraphics[scale=0.45]{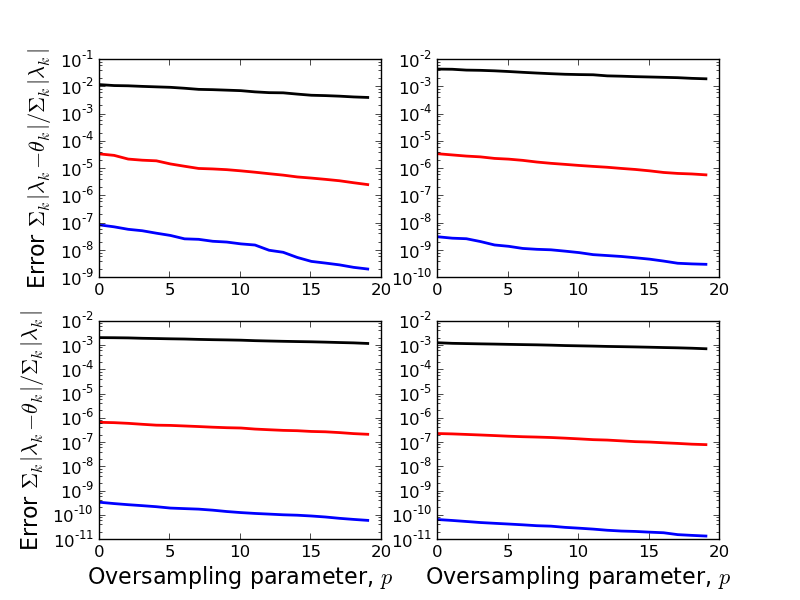}
\caption{Comparison of the error in the approximation of $T$ (that is formed in the two-pass algorithm) and its approximation $\tilde{T}$ (that is formed in the single-pass algorithm) measured as $\sum_k|\lambda_k-\theta_k|/\sum_k|\lambda_k|$ (where $\lambda_k$ and $\theta_k$ are defined in Theorem~\ref{thm:singlepasserr})as a function of oversampling parameter $p$ for each of the covariance kernels defined in Equation~\eqref{eqn:matern3} - $\kappa_{1/2}$ (black), $\kappa_{3/2}$ (red) and $\kappa_{5/2}$ (blue). The plots correspond to $k=20,40,60$ and $80$. }
\label{fig:errsinglepass}
\end{figure}

\begin{figure}[!ht]
\centering
\includegraphics[scale=0.45]{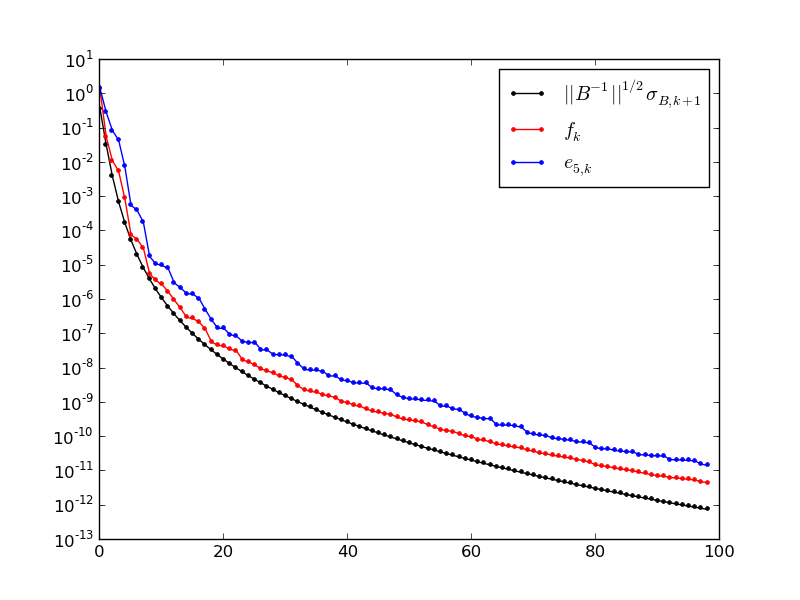}
\caption{Comparison of the actual error in the low-rank representation $f_k$ with the random estimator $e_k$ from Proposition~\ref{prop:random} and the approximation error from Theorem~\ref{thm:main}. An oversampling factor of $5$ is used and we also choose $r=5$ for the randomized estimator. Here, we use $\kappa_{5/2}$ defined in Equation~\eqref{eqn:matern3}.}
\label{fig:errsvd}
\end{figure}

\subsubsection{Accuracy of the estimator:} Next, we analyze the performance of the proposed estimator for the error in the low rank decomposition $\normB{(I-QQ^*B)C} \leq \varepsilon$. An oversampling factor of $p=5$ was used. We compare the following quantities:
\begin{itemize}
\item $\sqrt{\normtwo{B^{-1}}}\sigma_{B,k+1}(C)$, where $(k+1)$ generalized singular value of the matrix $C$. This is, roughly speaking, an estimate of the error according to Theorem~\ref{thm:main}.
\item The actual error in the low-rank approximation $f_k = \normB{(I-Q_kQ_k^*B)C}$.
\item Estimator of the error $f_k$ computed using the result in Proposition~\ref{prop:random}, and is denoted as $e_{5,k}$. We pick $\alpha = 2$ and $r=5$.
\end{itemize}

Figure~\ref{fig:errsvd} shows the comparison between the three quantities listed above. We observe that the error in the low-rank approximation $f_k$ is nearly equal to the estimate $\sqrt{\normtwo{B^{-1}}}\sigma_{B,k+1}(C)$ that is predicted from theory. Moreover, the true error is bounded from above by the estimated error $e_k$ and hence, the estimator provides a good upper bound for the actual error. Next, we try to answer the following question: How often (statistically speaking) is the estimator for the error close to the true error? To answer this, we generate $1000$ realizations at different values of $k=20,40,60,80$ and compare the true error with the estimated error. The results are presented in Figure~\ref{fig:errdist}. It can be seen that both the actual and the estimated error are concentrated about the mean.

\begin{figure}[!ht]
\centering
\includegraphics[scale=0.45]{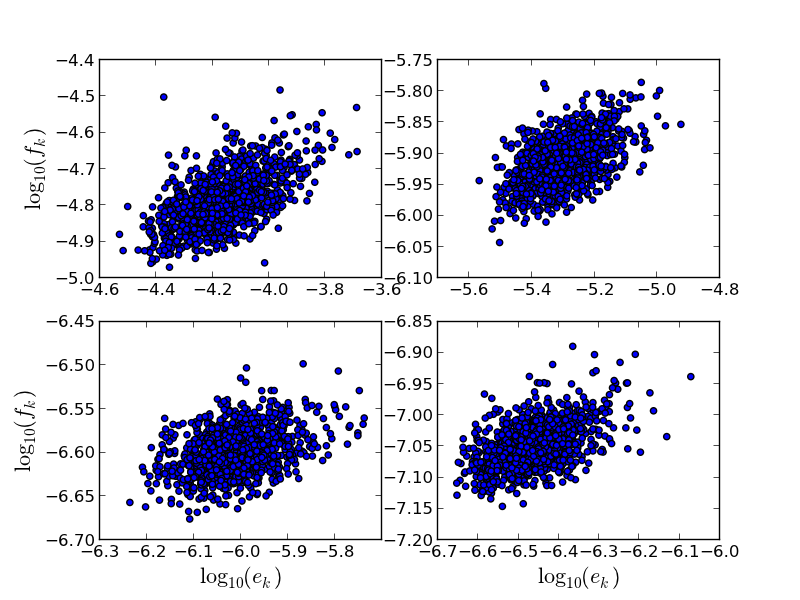}
\caption{Distribution of the true and the estimated error generated for $1000$ samples corresponding to $k=20,40,60,80$ eigenvalues. Here, $f_k$ is the  actual error in the low-rank representation  and the random estimator $e_k$ from Proposition~\ref{prop:random}. We use $\kappa_{3/2}$ defined in Equation~\eqref{eqn:matern3}. }
\label{fig:errdist}
\end{figure}

\subsubsection{Effect of correlation length $l$:} The rate of decay of eigenvalues is controlled by the smoothness of the kernel~\cite{schwab2006karhunen}. Additionally, the rate of decay is also dependent on the correlation length $l$, which appears in Equation~\eqref{eqn:matern3} through the distance function $d = \normtwo{\bx -\textbf{y}} / l$. As has been observed in~\cite{cliffe2011multilevel}, for small correlation lengths there is a pre-asymptotic regime before there is a significant decay rate of the eigenvalues. To demonstrate the effect of correlation length $l$ on the accuracy of the randomized calculations, we consider the following numerical experiment. The eigenvalues are computed for the KLE using the covariance kernel $\kappa_{\nu=5/2}$ as defined in Equation~\eqref{eqn:matern3}.  The domain for the computations is  $[-1,1]$ and the number of grid points are $501$. The true and approximate eigenvalues are displayed in Figure~\ref{fig:corrlength} for 3 different correlation lengths $l=[0.01,0.1,1]$. Also plotted is the error between the true and approximate eigenvalues measured as $\sum_k|\lambda_k-\tilde\lambda_|/\sum_k |\lambda_k|$. From the figure, it can be seen that there is no appreciable decay in the eigenvalues for small correlation lengths $0.5\%$ of domain length. However, for correlation lengths that are greater than $5\%$ of the domain length, which is typically used in practice, the accuracy of the eigenvalue calculations is moderate and improves significantly with increasing correlation length. It should be noted that the randomized algorithms may not be very accurate for extremely small correlation lengths.

\begin{figure}[!ht]
\centering
\includegraphics[scale=0.29]{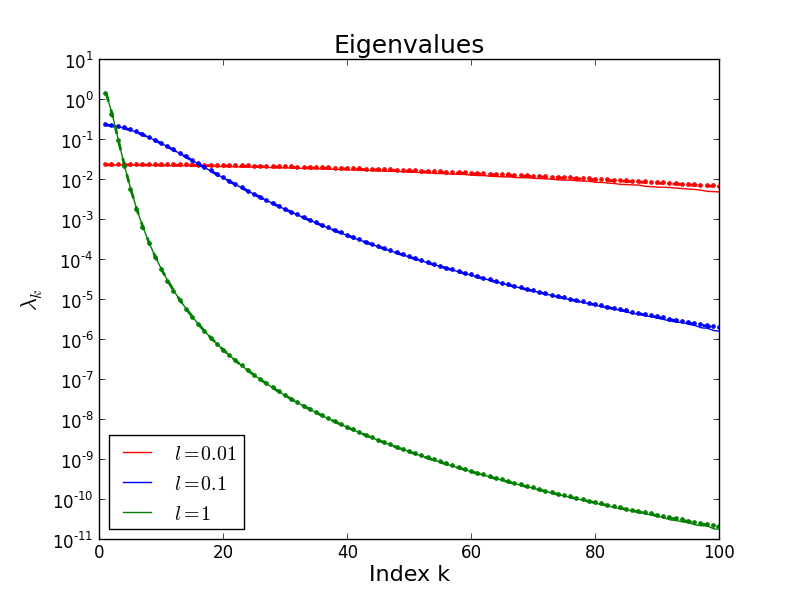}
\includegraphics[scale=0.29]{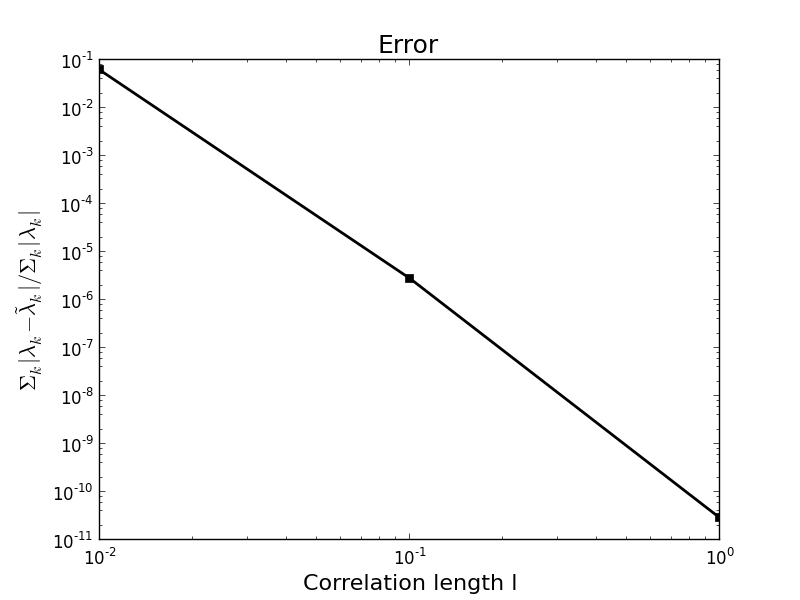}
\caption{Effect of correlation length on accuracy of eigenvalues.  We use $\kappa_{5/2}$ as the covariance kernel as defined in Equation~\eqref{eqn:matern3}. (left) comparison between the true eigenvalues (solid line) and approximate eigenvalues (dot-dashed line) computed for different correlation lengths (right) the error between the true and the approximate eigenvalues measured as $\sum_k|\lambda_k-\tilde\lambda_|/\sum_k |\lambda_k|$ as a function of correlation length. }
\label{fig:corrlength}
\end{figure}•

\subsubsection{Accuracy of the KL expansion:} Thus far, we have established the accuracy of the eigenvalues using the randomized approach. However, the accuracy of the KL expansion depends on both the accuracy of the eigenvalues and the eigenvectors. The accuracy of the truncated discrete KL expansion can be quantified using the following theorem.

\begin{theorem}\label{thm:kleerror}
Let $(\lambda,\phi)$ be the exact eigenpair of Equation~\eqref{eqn:intkle} and the $(\tilde\lambda,\tilde\phi)$ be the approximate eigenpair computed using the Randomized algorithms. Assume that $\arcsin (2\varepsilon/\delta) < \pi/2$
\[ \mathbb{E}\left[\left\lVert \sum_{k=1}^n \xi_k\left( \sqrt{\lambda_k}\phi -  \sqrt{\tilde\lambda_k} \tilde\phi\right) \right\rVert_M^2 \right]  \quad \lessapprox \quad  n\min\left\{2\varepsilon,\frac{2\varepsilon}{\delta}  \right\} + \sum_{k=1}^n \lambda_k \frac{4\varepsilon^2}{\delta^2}\]
Here the expectation $\mathbb{E}[\cdot]$ is w.r.t to the random variables $\xi_k$.
\end{theorem}•
\begin{proof}Using the property that $\mathbb{E}[\xi_i\xi_j] = \delta_{ij}$ the expression on the left reduces to
\[
  \mathbb{E}\left[\left\lVert \sum_{k=1}^n \xi_k\left( \sqrt{\lambda_k}\phi -  \sqrt{\tilde\lambda_k} \tilde\phi\right) \right\rVert_M^2 \right] = \sum_{k=1}^n  \left\lVert\sqrt{\lambda_k}\phi -  \sqrt{\tilde\lambda_k} \tilde\phi \right\rVert_M^2 \]
Next considering each term in the summation, we have
\begin{align*}
\left\lVert\sqrt{\lambda_k}\phi_k -  \sqrt{\tilde\lambda_k} \tilde\phi_k \right\rVert_M^2 \leq&  \quad \left\lVert\sqrt{\lambda_k}\phi_k -  \sqrt{\lambda_k} \tilde\phi_k \right\rVert_M^2 + \left\lVert\sqrt{\lambda_k}\tilde\phi_k -  \sqrt{\tilde\lambda_k} \tilde\phi_k \right\rVert_M^2 \\
\leq & \quad \lambda_k \norm{\phi_k - \tilde\phi}{M}^2 + |\lambda_k - \tilde\lambda_k| \norm{\tilde\phi}{M}^2
\end{align*}•
We have that $\norm{\tilde\phi_k}{M}^2 = 1$ and $\norm{\phi_k}{M}^2 = 1$ and $\angle_M(\phi_k,\tilde\phi_k) = \arccos \langle \phi, \tilde\phi\rangle_M$. 
\begin{align*}
\norm{\phi_k-\tilde\phi_k}{M}^2 \quad \leq & \quad \norm{\phi_k}{M}^2 + \norm{\tilde\phi_k}{M}^2 - 2\langle \phi, \tilde\phi_k\rangle_M = 2(1-\cos \angle_M(\phi_k,\tilde\phi_k))\\
= & \quad 2 \left( 1 - \sqrt{1 - \sin^2 \angle_M(\phi_k,\tilde\phi_k)}  \right)  \\
\leq &  \quad \left(\frac{2\varepsilon}{\delta}\right)^2  + \mathcal{O} \left(\frac{2\varepsilon}{\delta}\right)^4
\end{align*}•
Here we have used the result of Proposition~\ref{prop:apost} that bounds  $\sin\angle_M(\phi_k,\tilde\phi_k) \leq 2\varepsilon/\delta$. The proof is completed by plugging the above expression into the summation and using the inequality in Proposition~\ref{prop:apost}  $|\lambda_k-\tilde\lambda_k| \leq \min\{2\varepsilon,4\varepsilon^2/\delta \}$.
\end{proof}

\begin{figure}[!ht]
\centering
\includegraphics[scale=0.3]{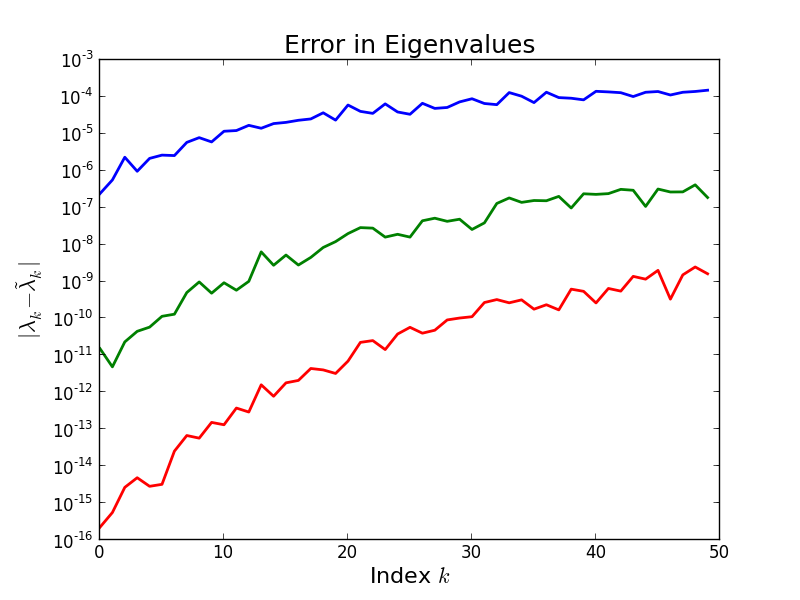}
\includegraphics[scale=0.3]{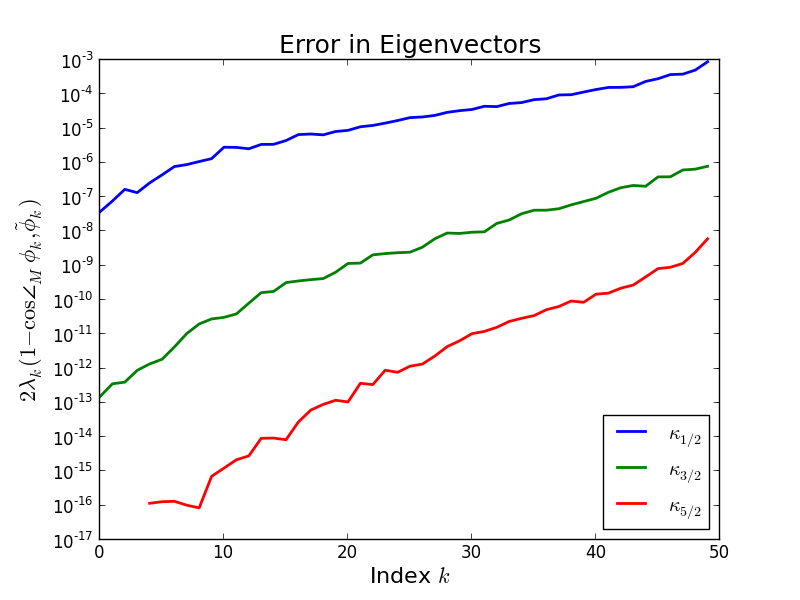}
\caption{(left) Accuracy of the eigenvalues  $\lambda_k - \tilde\lambda_k$ (right) accuracy of the eigenvectors quantified as $\lambda_k 2(1-\cos\angle_M(\phi_k,\tilde\phi_k))$ which appears in the proof of Theorem~\ref{thm:kleerror}. }
\label{fig:kleerror}
\end{figure}•

Estimation of the spectral gap is hard in practice, since the exact eigenvalues are not known. We consider the accuracy of the discretized KL expansion. We consider a 1D KL expansion in the domain $[-1,1]$ discretized using $501$ basis functions. Furthermore, we consider three different Mat\'ern class covariance kernels described in Equation~\eqref{eqn:matern3} and take the correlation length $l = 0.4$.  From the analysis in Theorem~\ref{thm:kleerror}, we have seen that the second factor controlling the error is the accuracy of the eigenvalues $|\lambda_k - \tilde\lambda_k|$  the factor $\lambda_k 2(1-\cos\angle_M(\phi_k,\tilde\phi_k))$ and these quantities have been plotted in Figure~\ref{fig:kleerror}.   From the figure, it can be seen that the error in both the quantities deteriorates with the index number of the eigenvalues $k$ and the accuracy is higher as the parameter $\nu$ increases. Furthermore, the accuracy of both quantities is roughly the same order of magnitude and therefore, both terms have similar contributions to the error in the discretized KL expansion computed using the randomized algorithms described in this paper.

\subsection{Implementation using $\mathcal{H}$-matrices}
Since the matrix $\prior$ is dense, storage and computational costs of matvecs involving the matrix $\prior$ scales as $\bigO(N^2)$. In order to mitigate these costs, ${\cal H}$-matrix approach has previously been used for efficient representation of covariance matrices arising out of Gaussian random fields in~\cite{saibaba2012efficient,ambikasaran2012large,saibaba2012application}. Hierarchical matrices~\cite{borm2003introduction} (or ${\mathcal{H}}$-matrices, for short) are efficient data-sparse representations of certain densely populated matrices. The main idea that is used repeatedly in these kind of techniques, is to split a given matrix into a hierarchy of rectangular blocks and approximate each of the blocks by a low-rank matrix. Hierarchical matrices have been used successfully in data-sparse representation of matrices arising in the Boundary Element method or for the approximation of the inverse of a Finite Element discretization of an elliptic partial differential operator. Fast algorithms have been developed for this class of matrices, including matrix-vector products, matrix addition, multiplication and factorization in almost linear complexity~\cite{borm2003introduction}. The matrix-vector products involving the dense covariance matrix can be computed in $\bigO(N\log N)$ using the ${\cal H}$-matrix approach, where $N$ is the number of grid points after discretization. The use of ${\cal H}$-matrices for computing the KLE along with Krylov subspace methods to compute eigendecomposition has been discussed in~\cite{khoromskij2009application,eiermann2007computational}. The specific details of our implementation of ${\cal{H}}$-matrix approach has already been presented in~\cite{saibaba2012application} and will not be provided here.

The assembly of the finite element matrix corresponding to the mesh is handled using the finite element software FEniCS~\cite{LoggWells2010a}. Since the matrix $M$ is sparse and can easily be factorized, computing the dominant eigenmodes of the eigenvalue problem~\eqref{eqn:klediscrete} can be efficiently computed by a transformation into a HEP. Instead, we only assume that $Mx$ and $M^{-1}x$ can be formed fast. We use this simple example to demonstrate the accuracy and speedup of the randomized algorithm for GHEP. We compare the performance of Algorithm~\ref{alg:doublepass} which is labeled ``Two Pass'', Algorithm~\ref{alg:singlepass} labeled ``Single Pass'' and the solution of the GHEP using ARPACK that is accessed via SciPy and is labeled ``ARPACK''. We warn the reader to exercise caution while interpreting the timing values, since the comparison is made across different programming environments (ARPACK is written in Fortran). Furthermore, any comparison with Krylov subspace methods is complicated by the fact that these methods require sophisticated algorithms for monitoring convergence and restarts.

\begin{figure}[!ht]
\centering
\includegraphics[scale = 0.25]{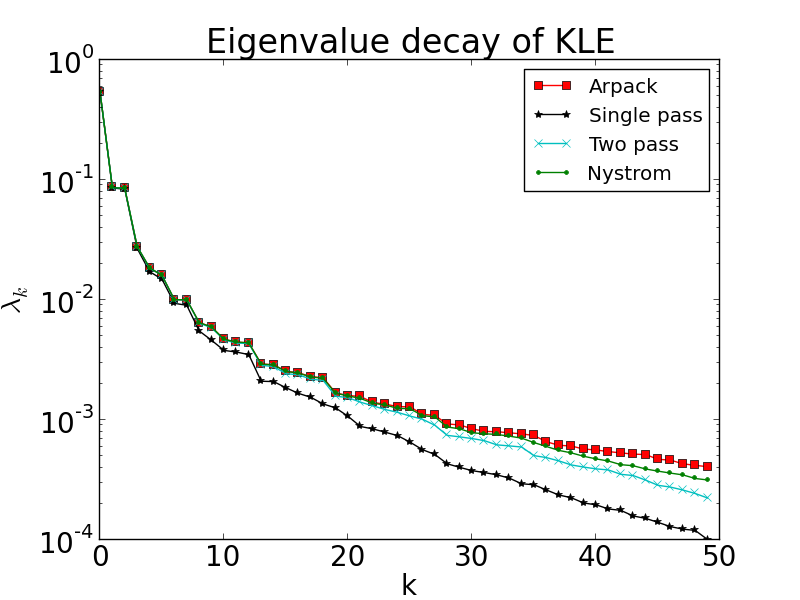}
\includegraphics[scale = 0.25]{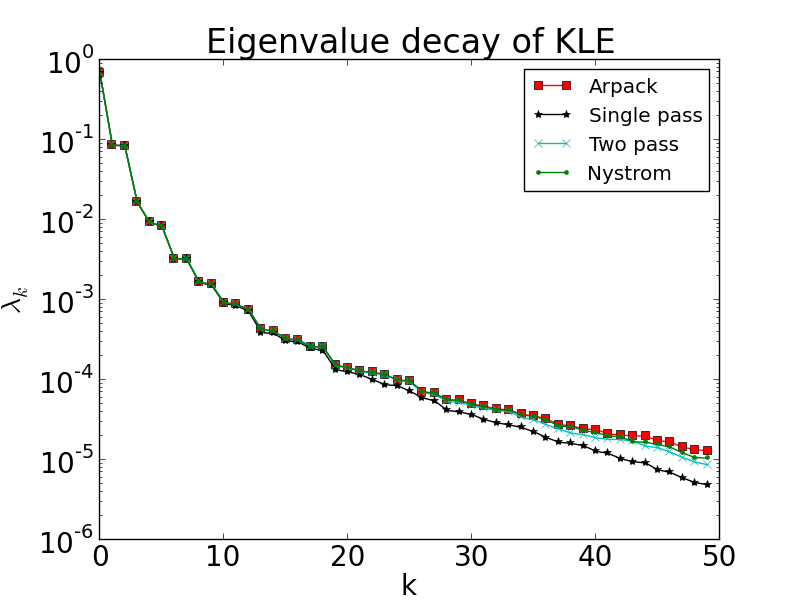}
\includegraphics[scale = 0.25]{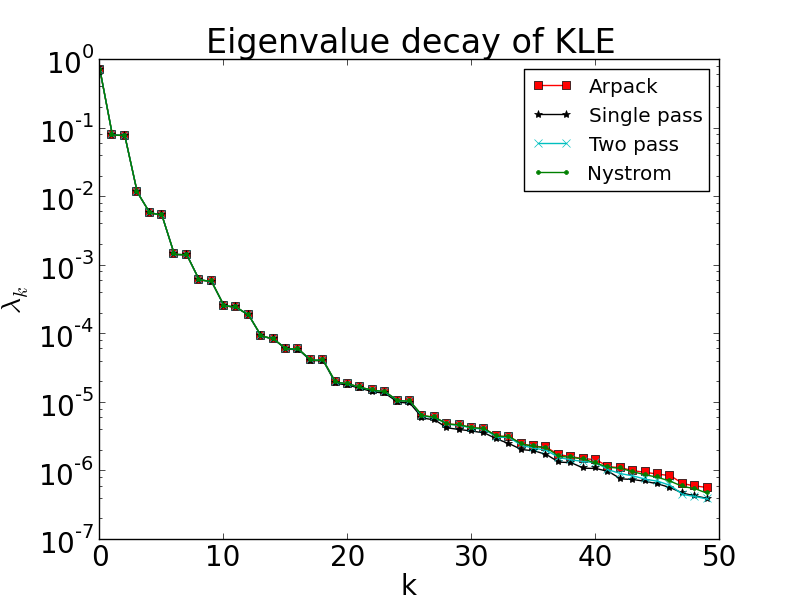}
\caption{Decay in eigenvalues corresponding to the covariance kernels defined in~\eqref{eqn:matern3} }
\label{fig:kleeigen}
\end{figure}

\begin{figure}[!ht]
\centering
\includegraphics[scale = 0.35]{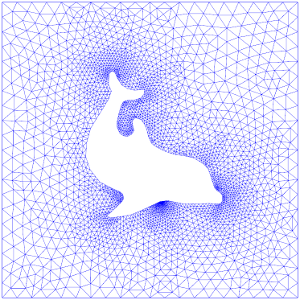}
\caption{Starting mesh used to illustrate the performance of the randomized algorithms to compute KLE.}
\label{fig:dolfin}
\end{figure}

The GHEP listed in Equation~\eqref{eqn:klediscrete} is solved corresponding to the covariance kernels defined in Equation~\eqref{eqn:matern3}. For the mesh, we started with a mesh available in the public domain \footnote{http://fenicsproject.org/download/data.html}. Then using the FEniCS command `refine' twice, we ended up with a finer mesh  with $43872$ nodes corresponding to the irregular domain in Figure~\ref{fig:dolfin}. The time to compute the eigendecomposition, and a summary of the number of matrix-vector products taken by each solver is summarized in Table~\ref{tab:eigen}. An oversampling factor $p=5$ is chosen. The eigenvalues are shown in Figure~\ref{fig:kleeigen}. We observe that even though the single pass algorithm takes half the time as the two pass algorithm (fewer matvecs with $A$), the accuracy of this algorithm deteriorates as the number of requested eigenvalues increase. The difference in computational costs between the randomized algorithms and Krylov subspace based eigensolvers will be much higher if the cost associated with forming $Bx$ or $B^{-1}x$ is much higher. In terms of the accuracy of the eigensolvers, we observe that in general, the ``Nystr\"om'' and ``Two Pass'' algorithms are closer in accuracy compared to the ``ARPACK'' solver; in fact, they are often an order of magnitude more accurate than the ``Single Pass'' algorithms. The accuracy improves when the eigenvalues decay more rapidly, i.e. for covariance matrices $\kappa_\nu$ with increasing $\nu$. Furthermore, we observe that the first few eigenvalues are computed relatively accurately but the accuracy decays towards the tail. This accuracy can be improved by increased oversampling, i.e. using a higher value of $p$. The summary of the computational costs along with CPU time is provided in Table~\ref{tab:eigen}.

\begin{table}[!ht]\centering
 \begin{tabular}{|c|c|c|c|c|c|}\hline
  Method &  $Ax$ & $Bx$ & $B^{-1}x$ & Time (s) & $\sum_k|\lambda_k-\tilde\lambda_k|/\sum_k|\lambda_k|$ \\ \hline
\multicolumn{6}{|c|}{$\kappa_{1/2}(r) = \exp(-r) $} \\ \hline
  Single Pass & $55$ & $157$ & $60$ & $91.49 $ & $3.6\times 10^{-2}$ \\
  Two Pass   & $110$ & $156$ & $55$ & $186.50$ & $7.0\times 10^{-3}$\\
  Nystr\"{o}m & $110$ & $157$ & $157$ & $188.32$&  $2.4\times 10^{-3}$  \\
  ARPACK     &  $ 128$ & $ 256 $& $128$ & $205.37$ &  $-$\\ \hline
\multicolumn{6}{|c|}{$\kappa_{3/2}(r) =  (1+\sqrt{3}r)\exp(-\sqrt{3}r)  $} \\ \hline
  Single Pass & $55$ & $160$ & $55$ & $ 95.40$ & $1.0\times 10^{-3}$ \\
  Two Pass   & $110$ & $162$ & $55$ & $186.68$ & $1.1\times10^{-4} $\\
  Nystr\"{o}m & $110$ & $160$ &  $160$ & $177.70$ & $3.5\times 10^{-5}$\\
  ARPACK     &  $ 102$ & $ 202 $& $102$ & $159.07$ &  $-$\\ \hline
\multicolumn{6}{|c|}{$\kappa_{5/2}(r) =   (1+\sqrt{5}r + \frac{5}{3}r^2)\exp(-\sqrt{5}r)  $} \\ \hline
  Single Pass & $55$ & $161$ & $55$ & $86.25 $ & $3.39\times 10^{-5}$ \\
  Two Pass   & $110$ & $162$ & $55$ & $ 171.72$ &  $4.31\times 10^{-6} $\\
  Nystr\"{o}m & $110$ & $162$ & $162$ & $172.64$ & $1.8\times 10^{-6}$ \\
  ARPACK     &  $ 102$ & $ 201 $& $102$ & $ 155.91$ & $-$\\ \hline
 \end{tabular}
 \caption{Comparison of computational costs of the randomized algorithms ``Single Pass'' (Algorithm~\ref{alg:singlepass}), ``Two Pass'' (Algorithm~\ref{alg:doublepass}) and ``Nystr\"om'' (Algorithm~\ref{alg:nystrom}) with ARPACK which is a standard eigensolver for GHEP. Here, the eigenvalues computed using ARPACK are treated as the ``true'' eigenvalues $\lambda_k$. }
 \label{tab:eigen}
\end{table}

Finally, we conclude this section with a discussion on choosing between randomized algorithms and Krylov subspace methods for computing the dominant eigenmodes of the KLE. As can be seen from Table~\ref{tab:eigen}, in general the single-pass algorithm is nearly twice as cheap compared to either two-pass algorithm or ARPACK since the dominant cost is forming matvecs with $A$. Although, on the whole two-pass algorithm is more accurate than the single-pass algorithm, in this application it is more expensive than ARPACK. Therefore, if an accurate eigendecomposition is desired then ARPACK is recommended. 

In finely discretized problems with complicated geometries in 3D, factorizing or inverting the mass matrix $M$ that is required by both randomized algorithms and ARPACK prove to be expensive. In such cases, the calculations can be simplified by observing that $B^{-1}A = \prior M$ and as a result, there is no reason to invert $M$. This can be used to accelerate the randomized algorithms. The same trick can be used by Krylov subspace methods as well. The ultimate choice of algorithms would depend heavily on the architecture used, the specific problem and the desired accuracy.

%Subsection on parallel implementation
%\include{randomized_parallel_subsection}
\subsection{Parallel implementation}
In this section, we consider the parallel performance of the proposed algorithms for a large-scale KL expansion. The domain $x \in [0, 1]^3$ was discretized with uniformly distributed $N = 200^3 = 8,000,000$ grid points. The computations were performed  for $k = 120$ eigenmodes with an oversampling factor $p = 8$. Parallel execution times were measured on a Linux workstation equipped with Intel Xeon E5-2687W running at 3.1 GHz (16 cores) and 128 GB memory. MATLAB was used to test single pass (Algorithm~\ref{alg:doublepass}) and two pass (Algorithm~\ref{alg:singlepass}) algorithms. Since the domain under consideration is a rectangle and the covariance kernel is stationary, the resulting covariance matrix is a recursive block-Toeplitz matrix and the dense matrix-vector products involving the matrix $\prior$ were accelerated using FFT~\cite{ambikasaran2013fast}.

For the QR decomposition with weighted inner-products, we consider the `PreCholQR' (Algorithm~\ref{alg:precholqr}) instead of `MGS-R' (Algorithm~\ref{alg:mgs}). The reason for this is that, like Krylov subspace methods, MGS-R uses $W$-inner products in a sequential fashion. On the other hand, the computations in `PreCholQR' can be readily parallelized. The matvecs $B^{-1}A = \prior M$ are further parallelized further simply using MATLAB command `parfor'. This convenient parallel implementation underscores the coding efficiency and excellent scalability of the randomized algorithm since typical Krylov subspace methods have to execute matrix-vector multiplications sequentially. The matrix $M$ was constructed using FEniCS. Up to 16 processes were used for the tests and each test was executed 10 times to compute the average execution time.
\begin{figure}[!ht]
\centering
\includegraphics[scale = 0.3]{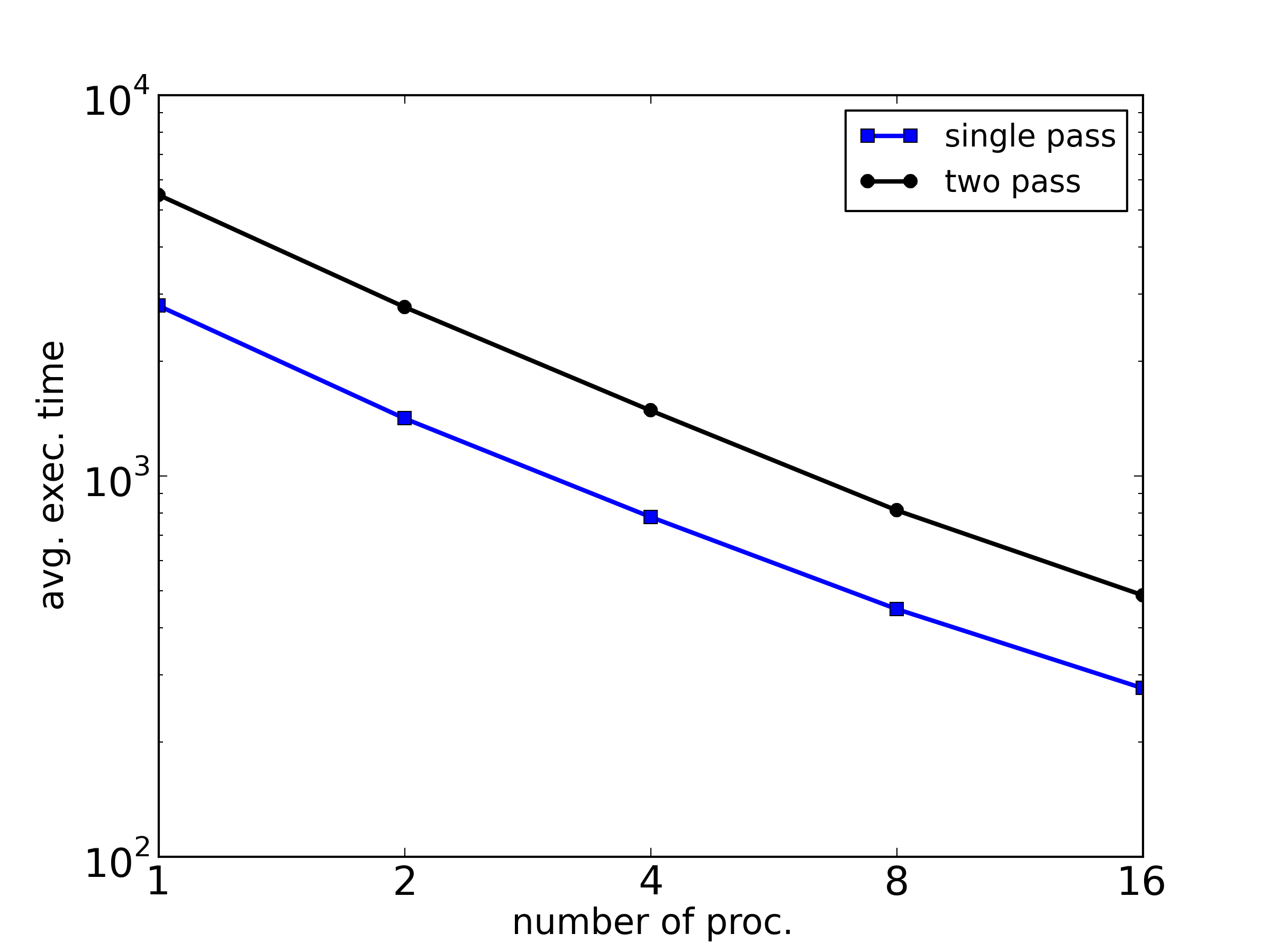}
\includegraphics[scale=0.45]{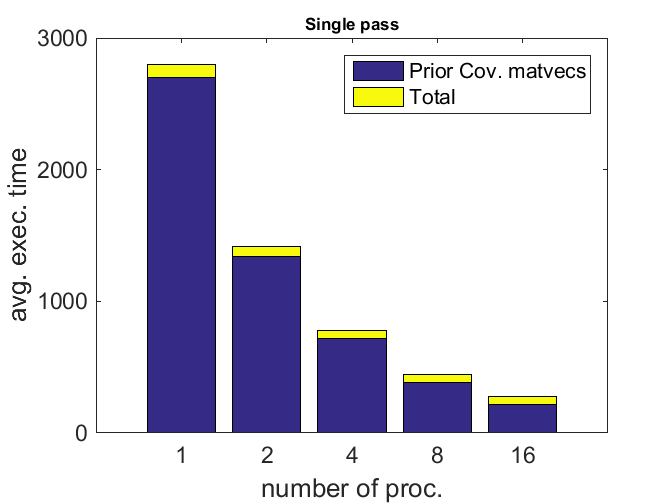}

\caption{(left) Performance results with 16 processes of randomized algorithms Two-pass  (Algorithm~\ref{alg:singlepass}) and Single pass (Algorithm~\ref{alg:doublepass}) for $k=120$ eigenvalues and oversampling factor $p = 20$. (right) Breakdown of costs of different parts of the algorithm, demonstrating the parts that are scalable. }
\label{fig:parallel}
\end{figure}

% add "\usepackage{multirow}" at the beginning!

\begin{table}\centering
\begin{tabular}{|c|c|c|c|c|}
\hline
\multirow{2}{*}{np} & \multicolumn{2}{|c|}{N = 125,000} & \multicolumn{2}{|c|}{N = 1,000,000} \\
\cline{2-5}
    &  randomized & eigs & randomized & eigs \\
\hline
 1  & 20.15 &        & 199.64 &       \\
 2  & 10.63 &        & 102.53 &       \\
 4  &  6.99 &  30.98 \textsuperscript{*}  &  59.05 & 254.56 \textsuperscript{*}\\
 8  &  4.18 &        &  35.86 &       \\
 16 &  3.34 &        &  25.94 &       \\
\hline
\multicolumn{5}{l}{\textsuperscript{*}\footnotesize{multithreaded on 16 cores}}
\end{tabular}
 \caption{Comparison of computational times (in seconds) of the randomized algorithm ``Single Pass'' (Algorithm~\ref{alg:singlepass}) with MATLAB function ``eigs'' for $N = 50^3 = 125,000$ and $100^3 = 1,000,000$ with k = 120 eigenmodes and p = 8 oversampling factor.}
 \label{tab:parallel}
\end{table}

Figure~\ref{fig:parallel} shows the strong scaling for single pass and double pass algorithms. In these experiments, the dominant computation cost arose from matrix-vector products $B^{-1}A = \prior M$ and simple embarrassingly parallel implementation on this step could reduce the overall computation costs significantly without losing its scalability. The remaining steps with smaller computation costs were executed using built-in MATLAB functions (sparse matrix multiplication $M*x$, chol and eig), which took around a minute on a single core. While one might reduce the computation time further using sophisticated parallelization on the entire algorithm, the computation costs for these steps become negligible for large-scale truncated KL expansion problems. We expect similar scaling when applied to clusters with distributed memory system. Finally, we demonstrate significant performance gains over a Krylov subspace implemented on the same computational environment. Table~\ref{tab:parallel} shows the comparison of computational costs of Single pass algorithm (Algorithm~\ref{alg:singlepass}) applied to the GHEP $M\prior M x = \lambda M x$ with MATLAB function `eigs' applied to the matrix $\prior M$ for $N = 50^3 = 125,000$ and $100^3 = 1,000,000$ with $k = 120$ eigenmodes and p = 8 oversampling factor. As can be seen there is significant speed up in using the randomized approach even on a small problem size. It should be noted that the spectrum of the eigenvalue problem $\prior Mx = \lambda x $ is identical to the GHEP $M\prior M x = \lambda M x $, however the eigenvectors obtained using eigs are not $M$-orthonormal.

\section{Discussion and conclusions}
We have presented a few algorithms for computing the dominant eigenmodes of the generalized Hermitian eigenvalue problem $Ax =\lambda Bx$ using a randomized approach. The algorithms avoid the need to factorize $B$ (or form products with $B^{1/2}$ or its inverse). This is advantageous for certain classes of problems, where factorizing $B$ is computationally expensive. Instead, we provide a Hermitian low-rank decomposition by using $B$-inner products. We discussed various issues related to computational costs and factors controlling accuracy through an example application that involved computing the dominant eigenmodes of the Karhunen-Lo\`{e}ve expansion. Out of the two algorithms proposed - although single pass algorithms are faster (on account of using half the number of matvecs with $A$), the accuracy that it provides may not be satisfactory unless the eigenvalues decay very rapidly.

We conclude with an additional example application, in which we think randomized algorithms maybe computationally beneficial. Consider a linear inverse problem of estimating parameters $s\in\mathbb{R}^{n_s}$ from noisy measurements $y \in \mathbb{R}^{n_y}$ with $n_y \ll n_s$. Using a Bayesian approach to recover the unknowns from the measurements, often one has to solve the following regularized least-squares problem
\[\hat{s} =  \arg\min_s \quad \norm{y-Hs}{\noise^{-1}}^2 + \norm{s-\mu}{\prior^{-1}}^2  \]

In addition to computing the best estimate $\hat{s}$, we would like to derive an efficient representation for the posterior covariance matrix $\post \define (H^T\noise^{-1}H + \prior^{-1})^{-1}$ since this gives us insight about quantifying the predictive uncertainty. For examples, the diagonals of the posterior covariance matrix $\post$ is related to the variance of the estimate. As before, $\prior$ is approximated as a ${\cal{H}}$-matrix which can be used form fast products of the form $\prior x$ and $\prior^{-1}x$ (using a Krylov subspace method). Forming and storing the posterior covariance matrix entry wise using the formula $\left (\prior^{-1} + H^T\noise^{-1}H\right)^{-1}$ is still out of the question. We consider the generalized Hermitian eigenvalue problem
\begin{equation}\label{eqn:postghep} H^T\noise^{-1}Hu  = \lambda \prior^{-1}u\end{equation}
Using any of the randomized algorithms described previously, we get the decomposition
\[ H_\text{data} \define  H^T\noise^{-1}H \quad \approx \quad  \prior^{-1}U_k \Lambda_k U_k^T \prior^{-1} \]
where the columns of the matrix $U$ are the generalized eigenvectors and $\Lambda_k$ is a diagonal matrix with entries as the generalized eigenvalues. Plugging this decomposition into the expression for $\post$, and applying the Woodbury identity
\[ \post = ( \prior^{-1}U \Lambda U^T \prior^{-1} + \prior^{-1})^{-1} = \prior - U_kD_kU_k^T + \bigO \left(\frac{\lambda_{k+1}}{1+\lambda_{k+1}} \right)\]
where, $D_k \define \text{diag}(\frac{\lambda_i}{1+\lambda_i})$. For several inverse problems the eigenvalues of the eigenproblem~\eqref{eqn:postghep} decay rapidly so that the low-rank approximation can be truncated for small $k$ resulting in an efficient representation of the posterior covariance matrix. We will discuss this application in an upcoming paper~\cite{saibaba2013uncertainty}.

\bibliographystyle{plain}
\bibliography{randomized}

\section{Appendix: Error estimation}
In this Section, we derive a probabilistic error for the low-rank approximation described in Theorem~\ref{thm:main}. The proof follows the arguments of~\cite{martinsson2011randomized,liberty2007randomized} closely and uses several key results of~\cite{halko2011finding}.

\begin{proof}
First, we derive a deterministic bound for $\normB{(I-QQ^*B)B^{-1}A}$. It can be shown that there exists a matrix $F$ such that
\[ \normB{(I-QQ^*B)C} \leq 2\normB{C - C\Omega F } + 2\normB{C\Omega G - QRG}\]
The proof of the above inequality follows~\cite{martinsson2011randomized,liberty2007randomized}
If we choose $Q$ and $R$ such that $C\Omega = QR$, the second term drops out. Such a $Q$ and $R$ can be constructed using Algorithm~\ref{alg:mgs}. Now, we show that for any matrix $C\define B^{-1}A$	and $\Omega$ with i.i.d. entries chosen from a  Gaussian distribution with zero mean and unit variance, there exists a matrix $G$ such that $C\Omega F$ is a good approximation to $C$ in B-norm. In fact, we show by construction that such an $F$ exists.

We denote the Generalized SVD of $C = U\begin{pmatrix}\Sigma_{B,1} & \\ & \Sigma_{B,2} \end{pmatrix}V^*$, where $\Sigma_{B,1}$ contain the $k$ largest singular values of $C$ in the generalized sense. For convenience, henceforth we drop the subscript $B$ on the singular values. Then,

\[  C\Omega G = U \begin{pmatrix}\Sigma_1 & \\ & \Sigma_2 \end{pmatrix} \begin{pmatrix}\Omega_1 \\ \Omega_2\end{pmatrix} F \]
where, we have $V^*\Omega = \begin{pmatrix}\Omega_1 \\ \Omega_2\end{pmatrix}$ is also a Gaussian random matrix, because they are invariant under rotation. Here $\Omega_1$ is $k\times (k+p)$ and $\Omega_2 $ is $(n-k)\times (k+p)$. Now, we choose $G\define [\Omega_1^ \dagger \quad 0 ]V^*$ so that
\begin{align*}
  C\Omega G & = \quad U \begin{pmatrix}\Sigma_1 & \\ & \Sigma_2 \end{pmatrix} \begin{pmatrix}\Omega_1 \\ \Omega_2\end{pmatrix} G \\ \nonumber
 & = U \quad \begin{pmatrix}\Sigma_1 & \\ & \Sigma_2 \end{pmatrix} \begin{pmatrix}\Omega_1 \\ \Omega_2\end{pmatrix} [\Omega_1^ \dagger \quad 0 ]V^* \\ \nonumber
 & = \quad U \begin{pmatrix}\Sigma_1 & \\ & \Sigma_2 \end{pmatrix}  \begin{pmatrix} I & 0\\ \Omega_2\Omega_1^\dagger & 0 \end{pmatrix} V^* \quad = \quad  U \begin{pmatrix} \Sigma_1 & 0\\\Sigma_2\Omega_2\Omega_1^\dagger &  0 \end{pmatrix} V^* \\ \nonumber
\end{align*}

Then, $ C - C\Omega G = U \begin{pmatrix} 0 & 0 \\ -\Sigma_2\Omega_2\Omega_1^\dagger & \Sigma_2 \end{pmatrix} V^*$
and applying matrix-norm inequalities (see Proposition~\ref{prop:random}), we have
\begin{align*}
 \normB{C-C\Omega G}^2 \quad  \leq & \quad  \normtwo{B^{-1}}\lVert B^{1/2}U \begin{pmatrix} 0 & 0 \\ -\Sigma_2\Omega_2\Omega_1^\dagger & \Sigma_2 \end{pmatrix} V^* \rVert^2  \\ \nonumber
\leq & \quad  \normtwo{B^{-1}}\normtwo{B^{1/2}U}^2\normtwo{V^*}^2 \left( \normtwo{\Sigma_2}^2 + \normtwo{\Sigma_2\Omega_2\Omega_1^\dagger}^2 \right) \\ \nonumber
=& \quad  \normtwo{B^{-1}}\left(\normtwo{\Sigma_2}^2  + \normtwo{\Sigma_2\Omega_2\Omega_1^\dagger}^2 \right) \\ \nonumber
\end{align*}
However, $\normtwo{\Sigma_2} = \sigma_{B,k+1}$. Now, applying the result in~\cite[Theorem 10.6]{halko2011finding} we get the desired result.

\end{proof}

\end{document}